\documentclass{amsart}
\usepackage{amsmath, amssymb, graphicx, amsthm}
\usepackage{pst-all}
\usepackage{xspace}

\numberwithin{equation}{section}

\theoremstyle{plain}
\newtheorem{prop}{Proposition}[section]
\newtheorem{coro}[prop]{Corollary}
\newtheorem{fact}[prop]{Fact}
\newtheorem{lemm}[prop]{Lemma}
\newtheorem{ques}[prop]{Question}

\newtheorem*{propA}{Proposition}

\theoremstyle{definition}
\newtheorem{defi}[prop]{Definition}
\newtheorem{rema}[prop]{Remark}
\newtheorem{exam}[prop]{Example}

\psset{unit=1mm}
\psset{dotsep=1.5pt}
\psset{dash=1.5pt 1.5pt}
\psset{linewidth=0.8pt}
\psset{arrowsize=3.5pt}
\psset{doublesep=1pt}
\psset{coilwidth=1.2mm}
\psset{coilheight=2}
\psset{coilaspect=20}
\psset{coilarm=2}
\newpsstyle{etc}{linestyle=dotted}
\newpsstyle{exist}{linestyle=dashed}
\newpsstyle{thin}{linewidth=0.5pt}
\newpsstyle{thinexist}{linewidth=0.5pt, linestyle=dashed}

\renewcommand\aa{a}
\newcommand\bb{b}
\newcommand\BP[1]{B_{#1}^{\scriptscriptstyle\pmb+}}
\newcommand\BPr[1]{\widetilde{B}_{#1}^{\scriptscriptstyle\pmb+}}
\newcommand\card[1]{\mathtt{\#}#1}
\newcommand\cc{c}
\newcommand\CC{C}
\newcommand\Cond{\diamondsuit}
\newcommand\Conds[1]{\diamondsuit^*_{\HS{-0.3}#1}}
\newcommand\dd{d}

\newcommand\distR{\mathrm{dist}_\RRR}
\newcommand\divR{\mathrel{\prec_{\scriptscriptstyle\HS{-0.3}R}}}
\newcommand\divL{\mathrel{\prec_{\scriptscriptstyle\HS{-0.3}L}}}
\newcommand\ee{e}
\newcommand\eqR{\equiv_\RRR}
\newcommand\ew{\varepsilon}
\newcommand\ff{f}
\let\ge=\geqslant
\renewcommand\gg{g}

\newcommand\hh{h}
\newcommand\HS[1]{\leavevmode\null\hspace{#1mm}}
\newcommand\ie{{\it i.e.}}
\newcommand\ii{i}

\newcounter{ITEM}
\newcommand\ITEM[1]{\setcounter{ITEM}{#1}\leavevmode\hbox{\rm(\roman{ITEM})}}
\newcommand\jj{j}
\newcommand\kk{k}
\let\le=\leqslant

\newcommand\MM{M}
\newcommand\MON[2]{\langle#1\,\vert\,\nobreak#2\rangle^{\scriptscriptstyle+}}

\newcommand\multR{\mathrel{\succ_{\scriptscriptstyle\HS{-0.6}R}}}

\newcommand\nn{n}

\newcommand\Ord{\mathbf{Ord}}
\newcommand\pdots{\mathbin{{\cdot}{\cdot}{\cdot}\HS{0.5}}}
\newcommand\pp{p}
\newcommand\qq{q}

\newcommand\rev{\curvearrowright}
\newcommand\revR{\rev_\RRR}
\newcommand\rr{r}
\newcommand\RRR{\mathcal{R}}
\newcommand\sig[1]{\sigma_{\!#1}^{\relax}}
\newcommand\sigg[2]{\sigma_{\!#1}^{\HS{-0.2}(\HS{-0.2}#2\HS{-0.2})\HS{-0.2}}}
\renewcommand\ss{s}
\newcommand\SSS{\mathcal{S}}
\newcommand\SSSb{\overline\SSS}
\newcommand\SSSs{\SSS^*}

\newcommand\tta{\mathtt{a}}

\newcommand\ttb{\mathtt{b}}

\newcommand\ttc{\mathtt{c}}

\newcommand\ttd{\mathtt{d}}

\newcommand\ttt{t}
\newcommand\uu{u}
\def\VR(#1,#2){\vrule width0pt height#1mm depth#2mm}

\newcommand\vv{v}

\newcommand\wdots{, ..., }
\newcommand\weq{\equiv}
\newcommand\ww{w}
\newcommand\xx{x}
\newcommand\yy{y}

\begin{document}

\author{Patrick DEHORNOY}

\address{Laboratoire de Math\'ematiques Nicolas Oresme, UMR 6139 CNRS, Universit\'e de Caen, 14032 Caen, France}
\email{patrick.dehornoy@unicaen.fr}
\urladdr{dehornoy.users.lmno.cnrs.fr}

\title{A cancellativity criterion for presented monoids}

\keywords{semigroup presentation, van Kampen diagram, rewrite system, cancellativity, word problem, Garside monoid, group of fractions, monoid embeddability, Artin--Tits groups}

\subjclass{20M05, 20M12, 20F36}

\begin{abstract}
We establish a new, fairly general cancellativity criterion for a presented monoid that properly extends the previously known related criteria. It is based on a new version of the word transformation called factor reversing, and its specificity is to avoid any restriction on the number of relations in the presentation. As an application, we deduce the cancellativity of some natural extension of Artin's braid monoid in which crossings are colored.
\end{abstract}

\maketitle

Establishing that a presented monoid (or semigroup) is cancellative is in general a nontrivial task, for which not so many methods are known~\cite[sec.\,5.3]{Hig}. If a distinguished expression (``normal form'') has been identified for each element of the monoid, and if, for each element~$\aa$ of the considered monoid~$\MM$ and every generator~$\ss$ of the considered presentation, the normal form of~$\aa$ can be retrieved from that of~$\ss\aa$ and~$\ss$, then one can indeed conclude that $\ss\aa = \ss\bb$ implies $\aa = \bb$. But, when no normal form is known, no generic method is available. Adjan's criterion based on the left graph~\cite{Adj, Rem} is useful, but, by definition, it applies only to presentations with (very) few defining relations. Ultimately relying on Garside's analysis of the braid monoids~$\BP\nn$~\cite{Gar}, the so-called reversing method~\cite{Dff, Dia} provides a simple criterion, which proved to be useful for many concrete presentations, typically those of all Artin--Tits monoids. However, an intrinsic limitation of the method is that it only applies to monoid presentations~$(\SSS, \RRR)$ that contain a limited number of relations, namely those such that, for all~$\ss, \ttt$ in~$\SSS$, there exists at most one relation of the form $\ss ... \,{=}\, \ttt ...$ in~$\RRR$ (``right-complemented'' presentations). The aim of this paper is to extend the previous criterion by developing a new approach that requires no limitation on the number of defining relations. The result we prove takes the following form:

\begin{propA}
Assume that $(\SSS, \RRR)$ is a monoid presentation such that 

\ITEM1 there exists an $\eqR$-invariant map~$\lambda$ from~$\SSSs$ to ordinals satisfying $\lambda(\ss\ww) >\nobreak \lambda(\ww)$ for all~$\ss$ in~$\SSS$ and~$\ww$ in~$\SSSs$, and

\ITEM2 for every~$\ss$ in~$\SSS$, for every relation~$\ww \,{=}\, \ww'$ in~$\RRR$, and for every $(\SSS, \RRR)$-grid from~$(\ss, \ww)$, there exists an equivalent grid from~$(\ss, \ww')$, and vice versa, and

\ITEM3 there is no relation~$\ss\ww \,{=}\, \ss\ww'$ in~$\RRR$ with $\ww, \ww'$ distinct. 

\noindent Then the monoid associated with~$(\SSS, \RRR)$ admits left cancellation.
\end{propA}

In the above statement, $\eqR$ refers to the congruence on the free monoid~$\SSSs$ generated by the relations of~$\RRR$, and an $(\SSS, \RRR)$-grid is a certain type of rectangular van Kampen diagram specified in Definition~\ref{D:Grid} below. Note that Condition~\ITEM1 in the above statement is trivial when each relation in~$\RRR$ consists of two words with the same length (``homogeneous presentation''), since, in that case, $\lambda(\ww)$ can be taken to be the length of~$\ww$. 

As an application, we deduce:

\begin{propA} 
For every~$\nn$ and every nonempty set~$\CC$, the monoid
\begin{equation*}
\BP{\nn, \CC}:= \bigg\langle \sigg\ii\aa ( \ii \le \nn, \aa \in \CC ) \ \bigg\vert\ 
\begin{matrix}
\sigg\ii\aa \sigg\jj\bb = \sigg\jj\bb \sigg\ii\aa 
&\text{for} &\vert i-j \vert\ge 2\\
\sigg\ii\aa \sigg\jj\bb \sigg\ii\cc = \sigg\jj\cc \sigg\ii\bb \sigg\jj\aa 
&\text{for} &\vert i-j \vert = 1
\end{matrix}
\ \bigg\rangle^{\!+}
\end{equation*}
is left and right cancellative.
\end{propA}

The monoid~$\BP{\nn, \CC}$ is an extension of Artin's monoid $\BP\nn$ of positive $\nn$-strand braids, and it is a typical example of a monoid that is inaccessible to all previously known methods. We shall see that the elements of~$\BP{\nn, \CC}$ admit a natural interpretation in terms of braids with $\CC$-colored crossings, and both $\BP{\nn, \CC}$ and its universal group might be structures of independent interest. They are in particular directly reminiscent of (but not identical to) the monoids investigated in~\cite{BGL}.

\section{Using reversing grids}\label{S:Criterion}

As in~\cite{Dia}, our cancellativity criterion is based on some word transformation called factor (or subword) reversing, and on a criterion for establishing that the latter is complete, meaning that it detects every word equivalence with respect to the presentation. The specificity of the current paper is to extend the framework so as to avoid any restriction on the number of relations in the presentation. This is done by introducing the new notion of a reversing grid (Section~\ref{SS:Grid}), then establishing a convenient completeness criterion (Section~\ref{SS:Complete}), and finally deducing the expected cancellativity criterion and various other consequences (Section~\ref{SS:Cancel}).

\subsection{The notion of a reversing grid}\label{SS:Grid}

If~$\SSS$ is a nonempty set, we denote by~$\SSSs$ the free monoid of all words in~$\SSS$, and use~$\ew$ for the empty word. A \emph{monoid presentation} is a pair~$(\SSS, \RRR)$, where $\RRR$ is a list of (unordered) pairs of words of~$\SSSs$; as usual, we write relations with an equality sign, thus writing $\ww \,{=}\, \ww'$ for~$\{\ww, \ww'\}$. We then denote by~$\MON\SSS\RRR$ the monoid presented by~$(\SSS, \RRR)$, that is, the monoid~$\SSSs{/}{\eqR}$, where $\eqR$ is the congruence on~$\SSSs$ generated by~$\RRR$. To avoid any confusion due to using~$=$ in relations of the presentation, we shall use~$\weq$ for word equality. 

 A relation of the form~$\ww \,{=}\, \ew$ with~$\ww$ nonempty will be called an \emph{$\ew$-relation}. In the sequel, we shall only address monoid presentations~$(\SSS, \RRR)$ that contain no $\ew$-rel\-ation. In this case, the only invertible element of the monoid~$\MON\SSS\RRR$ is the unit~$1$, represented by the empty word. Note that every such presentation also defines a semigroup and, in fact, most statements of this paper can be adapted to a semigroup context. However, the overall philosophy here is really that of monoids, and it seems more natural to stick to a monoid context. 

Our main subject of investigation is a certain binary relation (or rewrite system) on~$\SSS \times \SSS$ associated with~$(\SSS, \RRR)$ as follows.

\begin{defi}\label{D:Grid}
If $(\SSS, \RRR)$ is a monoid presentation, an \emph{$(\SSS, \RRR)$-grid} is a rectangular diagram consisting of finitely many matching $\SSS \cup \{\ew\}$-labeled pieces of the types

\VR(12,9)\begin{picture}(34,0)(-3,6)
\pcline{->}(1,13)(21,13)\taput{$\ttt$}
\pcline{->}(0,12)(0,1)\tlput{$\ss$}
\pcline{->}(1,0)(7,0)\taput{$\ttt_1$}
\pcline[style=etc](8,0)(14,0)
\pcline{->}(15,0)(21,0)\taput{$\ttt_\qq$}
\pcline{->}(22,12)(22,8)\trput{$\ss_1$}
\pcline[style=etc](22,7)(22,6)
\pcline{->}(22,5)(22,1)\trput{$\ss_\pp$}
\end{picture} 
\parbox{90mm}{with $\ss, \ttt, \ss_1 \wdots \ss_\pp, \ttt_1 \wdots \ttt_\qq$ in~$\SSS$\\ \null\hspace{20mm}and $\ss\ttt_1 {\pdots} \ttt_\qq = \ttt \ss_1 {\pdots} \ss_\pp$ a relation of~$\RRR$,} 

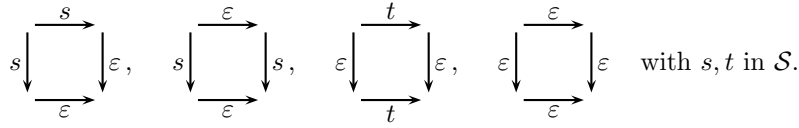
\VR(8,8)\begin{picture}(16,0)(-3,4)
\pcline{->}(1,10)(9,10)\taput{$\ss$}
\pcline{->}(0,9)(0,1)\tlput{$\ss$}
\pcline{->}(1,0)(9,0)\tbput{$\ew$}
\pcline{->}(10,9)(10,1)\trput{$\ew$}
\end{picture}, \quad
\begin{picture}(16,0)(-3,4)
\pcline{->}(1,10)(9,10)\taput{$\ew$}
\pcline{->}(0,9)(0,1)\tlput{$\ss$}
\pcline{->}(1,0)(9,0)\tbput{$\ew$}
\pcline{->}(10,9)(10,1)\trput{$\ss$}
\end{picture}, \quad
\begin{picture}(16,0)(-3,4)
\pcline{->}(1,10)(9,10)\taput{$\ttt$}
\pcline{->}(0,9)(0,1)\tlput{$\ew$}
\pcline{->}(1,0)(9,0)\tbput{$\ttt$}
\pcline{->}(10,9)(10,1)\trput{$\ew$}
\end{picture}, \quad
\begin{picture}(16,0)(-3,4)
\pcline{->}(1,10)(9,10)\taput{$\ew$}
\pcline{->}(0,9)(0,1)\tlput{$\ew$}
\pcline{->}(1,0)(9,0)\tbput{$\ew$}
\pcline{->}(10,9)(10,1)\trput{$\ew$}
\end{picture}\quad
with $\ss, \ttt$ in~$\SSS$.

\noindent For $\uu, \vv, \uu_1, \vv_1$ in~$\SSSs$, we say that an $(\SSS, \RRR)$-grid~$\Gamma$ goes \emph{from~$(\uu, \vv)$ to~$(\uu_1, \vv_1)$} or, equivalently, that $(\uu, \vv)$ is the \emph{source} of~$\Gamma$ and that $(\uu_1, \vv_1)$ is its \emph{target}, if the labels of the left and top edges of~$\Gamma$ form the words~$\uu$ and~$\vv$, respectively, whereas the labels of the right and bottom edges form the words~$\uu_1$ and~$\vv_1$. If there exists an $(\SSS, \RRR)$-grid from~$(\uu, \vv)$ to~$(\uu_1, \vv_1)$, we say that $(\uu, \vv)$ is \emph{right $\RRR$-reversible} to~$(\uu_1, \vv_1)$, written $(\uu, \vv) \rev_\RRR (\uu_1, \vv_1)$ or, in a diagrammatic way, \VR(9,6)\begin{picture}(27,0)(-3,4)
\pcline{->}(1,10)(19,10)\taput{$\vv$}
\pcline{->}(1,0)(19,0)\tbput{$\vv_1$}
\pcline{->}(0,9)(0,1)\tlput{$\uu$}
\pcline{->}(20,9)(20,1)\trput{$\uu_1$}
\put(7,4){$\rev_\RRR$}
\end{picture}; we then often skip~$\RRR$ if there is no ambiguity.
\end{defi}

 In the above definition, a grid consists of (finitely many) rectangular pieces (or tiles) whose edges are labeled by one or several letters of the current alphabet~$\SSS$ or by~$\ew$, and two adjacent pieces are called \emph{matching} if the letters occurring on the common part of their edges coincide (but we do not assume that all edges involve the same number of letters: by definition, there is only one letter on the top and the left edges, whereas there can be more on the bottom and right edges, depending on the length of the words involved in the relations of the presentation). Let us immediately observe that, by definition, a grid that contains more than one elementary piece can be split into the union of several grids, see for instance Lemma~\ref{L:Decomp} below.

\begin{exam}\label{X:Braid}
Consider the Artin presentation of the $\nn$-strand braid monoid
\begin{equation}\label{E:BraidPres}
\BP\nn:= \bigg\langle \sig1 \wdots \sig{\nn - 1} \ \bigg\vert\ 
\begin{matrix}
\sig\ii \sig j = \sig j \sig\ii 
&\text{for} &\vert i-j \vert\ge 2\\
\sig\ii \sig j \sig\ii = \sig j \sig\ii \sig j 
&\text{for} &\vert i-j \vert = 1
\end{matrix}
\ \bigg\rangle^{\!+}.
\end{equation}
For $\nn \ge 4$, a typical grid for~\eqref{E:BraidPres} is
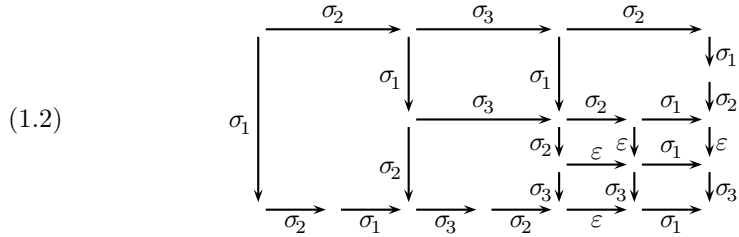
\begin{equation}\label{E:ExBraid}
\VR(16,15)\begin{picture}(60,0)(0,11)
\pcline{->}(1,24)(19,24)\taput{$\sig2$}
\pcline{->}(21,24)(39,24)\taput{$\sig3$}
\pcline{->}(41,24)(59,24)\taput{$\sig2$}
\pcline{->}(21,12)(39,12)\taput{$\sig3$}
\pcline{->}(41,12)(49,12)\taput{$\sig2$}
\pcline{->}(51,12)(59,12)\taput{$\sig1$}
\pcline{->}(41,6)(49,6)\taput{$\ew$}
\pcline{->}(51,6)(59,6)\taput{$\sig1$}
\pcline{->}(1,0)(9,0)\tbput{$\sig2$}
\pcline{->}(11,0)(19,0)\tbput{$\sig1$}
\pcline{->}(21,0)(29,0)\tbput{$\sig3$}
\pcline{->}(31,0)(39,0)\tbput{$\sig2$}
\pcline{->}(41,0)(49,0)\tbput{$\ew$}
\pcline{->}(51,0)(59,0)\tbput{$\sig1$}
\pcline{->}(0,23)(0,1)\tlput{$\sig1$}
\pcline{->}(20,23)(20,13)\tlput{$\sig1$}
\pcline{->}(20,11)(20,1)\tlput{$\sig2$}
\pcline{->}(40,23)(40,13)\tlput{$\sig1$}
\pcline{->}(40,11)(40,7)\tlput{$\sig2$}
\pcline{->}(40,5)(40,1)\tlput{$\sig3$}
\pcline{->}(50,11)(50,7)\tlput{$\ew$}
\pcline{->}(50,5)(50,1)\tlput{$\sig3$}
\pcline{->}(60,23)(60,19)\trput{$\sig1$}
\pcline{->}(60,17)(60,13)\trput{$\sig2$}
\pcline{->}(60,11)(60,7)\trput{$\ew$}
\pcline{->}(60,5)(60,1)\trput{$\sig3$}
\end{picture}
\end{equation}
It contains eight squares, of which five correspond to relations of~\eqref{E:BraidPres}, and it goes from~$(\sig1, \sig2 \sig3 \sig2)$ to~$(\sig1 \sig2 \sig3, \sig2 \sig1 \sig3 \sig2 \sig1)$, witnessing the right reversing relation $(\sig1, \sig2 \sig3 \sig2) \rev (\sig1 \sig2 \sig3, \sig2 \sig1 \sig3 \sig2 \sig1)$, alias \VR(10,8)\begin{picture}(36,0)(-5,4)
\pcline{->}(1,10)(19,10)\taput{$\sig2 \sig3 \sig2$}
\pcline{->}(1,0)(19,0)\tbput{$\sig2 \sig1 \sig3 \sig2 \sig1$}
\pcline{->}(0,9)(0,1)\tlput{$\sig1$}
\pcline{->}(20,9)(20,1)\trput{$\sig1 \sig2 \sig3$}
\put(8,4){$\rev$}
\end{picture}.
\end{exam}

In all five types of elementary pieces considered in Definition~\ref{D:Grid}, the labels of the two possible paths from the top-left vertex to the bottom-right vertex form words that are $\eqR$-equivalent, \ie, represent the same element in the monoid~$\MON\SSS\RRR$. An easy induction on the number of elementary pieces implies:

\begin{lemm}\label{L:RevEq}
For every monoid presentation~$(\SSS, \RRR)$, and for all words~$\uu, \vv, \uu_1, \vv_1$ in~$\SSSs$, the relation $(\uu, \vv) \rev_\RRR (\uu_1, \vv_1)$ implies $\uu\vv_1 \equiv_\RRR \vv \uu_1$. In particular, 
\begin{equation}\label{E:RevEq}
\text{$(\uu, \vv) \rev_\RRR (\ew, \ew)$ \quad implies \quad $\uu \equiv_\RRR \vv$.}
\end{equation}
\end{lemm}

In other words, a reversing grid from~$(\uu, \vv)$ to~$(\uu_1, \vv_1)$ is a special type of van Kampen diagram witnessing the $\equiv_\RRR$-equivalence of the words~$\uu\vv_1$ and~Ê$\vv\uu_1$.

\begin{rema}
In the articles \cite{Dff, Dgc, Dgp, Dhg}, reversing was described in terms of signed $\SSS$-words, defined to be words in a symmetrized alphabet~$\SSS \cup \SSSb$ with $\SSSb$ consisting of one copy~$\overline\ss$ for each letter~$\ss$ of~$\SSS$. If $\ww, \ww'$ are signed $\SSS$-words, one declares that $\ww \rev_\RRR \ww'$ holds if one can go from~$\ww$ to~$\ww'$ by a finite sequence of transformations, each of which consists either in deleting some length two factor~$\overline\ss \ss$, or in replacing some length two factor~$\overline\ss \ttt$ with $\ttt_1 {\pdots} \ttt_\qq \overline\ss_\pp {\pdots} \overline\ss_1$, where $\ss \ttt_1 {\pdots} \ttt_\qq = \ttt \ss_1 {\pdots} \ss_\pp$ is a relation of~$\RRR$. The connection with our current approach is easy: writing $\overline\ww$ for the word obtained from~$\ww$ by exchanging~$\ss$ and~$\overline\ss$ everywhere and reversing the order of letters, the relation $(\uu, \vv) \rev_\RRR (\uu_1, \vv_1)$ of Definition~\ref{D:Grid} is equivalent to $\overline\uu \vv \rev_\RRR \vv_1 \overline{\uu_1}$ in the sense of signed word reversing. The advantage of the current description is to make it more visible that reversing only involves positive words and the presented monoid, without connection with inverting the elements and moving to a group context. In any case, the reversing grid is the fundamental object, and it seems more natural to begin with it.
\end{rema}

One of the advantages of the current grid-based approach is to make the following technical result almost straightforward:

\begin{lemm}\label{L:Decomp}
For every monoid presentation~$(\SSS, \RRR)$, and for all~$\uu$, $\vv'$, $\vv''$, $\uu_1$, $\vv_1$ in~$\SSSs$, the following are equivalent:

\ITEM1 The relation $(\uu, \vv'\vv'') \rev_\RRR (\uu_1, \vv_1)$ holds; 

\ITEM2 There exist~$\uu', \vv'_1, \vv''_1$ in~$\SSSs$ satisfying $(\uu, \vv') \rev_\RRR (\uu', \vv'_1)$, $(\uu', \vv'') \rev_\RRR (\uu_1, \vv''_1)$, and $\vv_1 \weq \vv'_1 \vv''_1$.
\end{lemm}

\begin{proof}
Assume that $\Gamma$ is an $(\SSS, \RRR)$-grid from~$(\uu, \vv'\vv'')$ to~Ê$(\uu_1, \vv_1)$. By definition, $\Gamma$ is a juxtaposition of elementary diagrams as in Definition~Ê\ref{D:Grid}. 

\noindent\begin{minipage}{\textwidth}
\rightskip40mm \VR(3.2,0) Grouping the diagrams that lie below~$\vv'$ on the one hand, and below~$\vv''$ on the other hand, splits~$\Gamma$ into two grids~$\Gamma'$ and~$\Gamma''$. By construction, the input of~$\Gamma'$ is~$(\uu, \vv')$; call its output~$(\uu', \vv'_1)$. Then, by construction, the input of~$\Gamma'$ is~$(\uu', \vv'')$, and its output has the form~$(\uu_1, \vv''_1)$, with $\vv_1 \weq \vv'_1 \vv''_1$. So \ITEM1 implies~\ITEM2.\ \hfill
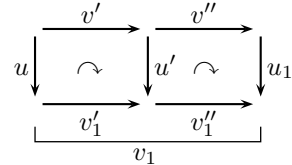
\begin{picture}(0,0)(-6,-10)
\pcline{->}(1,10)(14,10)\taput{$\vv'$}
\pcline{->}(16,10)(29,10)\taput{$\vv''$}
\pcline{->}(1,0)(14,0)\tbput{$\vv'_1$}
\pcline{->}(16,0)(29,0)\tbput{$\vv''_1$}
\pcline{->}(0,9)(0,1)\tlput{$\uu$}
\pcline{->}(15,9)(15,1)\trput{$\uu'$}
\pcline{->}(30,9)(30,1)\trput{$\uu_1$}
\psline[linewidth=0.5pt](0,-3)(0,-5)(30,-5)(30,-3)\put(13,-7.5){$\vv_1$}
\put(5.5,4){$\rev$}
\put(21,4){$\rev$}
\end{picture}
\end{minipage}

\VR(3.5,0) Conversely, concatenating a grid from~$(\uu, \vv')$ to~$(\uu', \vv'_1)$ and a grid from~$(\uu', \vv'')$ to~$(\uu_1, \vv''_1)$ provides a grid from~$(\uu, \vv'\vv'')$ to~$(\vv'_1\vv''_1, \uu_1)$, so \ITEM2 implies~\ITEM1. 
\end{proof}

\subsection{Completeness of reversing}\label{SS:Complete}

A reversing grid is a van Kampen diagram of a special type, namely one in which at most two edges (one horizontal, one vertical) start from each node. If there exists an $(\SSS, \RRR)$-grid from~$(\uu, \vv)$ to~$(\ew, \ew)$, then, by Lemma~\ref{L:RevEq}, the words~$\uu$ and~$\vv$ must be $\eqR$-equivalent. Conversely, if $\uu$ and~$\vv$ are $\eqR$-equivalent words, there must exist a van Kampen diagram connecting~$\uu$ and~$\vv$ but, in general, there is no reason why the latter could be chosen with the special form of a reversing grid: for instance, Fact~\ref{F:BGL} below will provide an example of an equivalence that cannot be recognised by a reversing grid. We now consider the case when using reversing grids is always possible.

\begin{defi}
We say that right reversing is \emph{complete} for a monoid presentation~$(\SSS, \RRR)$ if the converse of~\eqref{E:RevEq} also holds, that is, if, for all~$\uu, \vv$ in~$\SSS^*$, 
\begin{equation}\label{E:Complete}
\text{$(\uu, \vv) \rev_\RRR (\ew, \ew)$ \quad is equivalent to \quad $\uu \equiv_\RRR \vv$.}
\end{equation}
\end{defi}

This definition is theoretical, and our aim will be to establish a practical criterion characterizing completeness of reversing. Two such criteria have already appeared. A first criterion is described in~\cite{Dgp}, in terms of what is called the cube condition: in principle, this criterion works for arbitrary presentations but, in practice, it can be used only for complemented presentations, namely presentations with at most one relation $\ss ... \,{=}\, \ttt ...$ for each pair of generators~$(\ss, \ttt)$. Another criterion is described in~\cite{Dgc}, but, even in theory, it does not apply to presentations that are not complemented. What we do below is establish a new completeness criterion that extends the one of~\cite{Dgc} and works for every presentation, complemented or not. The main point is that this new criterion, contrary to the cube condition, remains tractable in the non-complemented case, \ie, without any restriction on the number of relations in the considered presentation. 

It follows from the definition of a reversing grid that reversing can be complete only for monoid presentations with no $\ew$-relation: indeed, by definition, $(\ww, \ew) \rev\nobreak (\ew, \ew)$ is impossible for~$\ww$ nonempty. So we shall hereafter restrict to monoid presentations with no $\ew$-relation. The crucial notion is then the notion of equivalent grids.

\begin{defi}
If $(\SSS, \RRR)$ is a monoid presentation with no $\ew$-relation, two $(\SSS, \RRR)$-grids~$\Gamma, \Gamma'$ are said to be \emph{equivalent} if the labels of the four edges of~$\Gamma$ form words that are $\eqR$-equivalent to their counterparts in~$\Gamma'$.
\end{defi}

We shall start from the following observation.

\begin{lemm}\label{L:CompatCompl}
If $(\SSS, \RRR)$ is a monoid presentation with no $\ew$-relation, a sufficient condition for right reversing to be complete for~$(\SSS, \RRR)$ is that, for all~$\uu, \vv$ in~$\SSSs$,
\begin{equation}\label{E:Compat}\tag{$\Conds{}$}
\parbox{113mm}{For every grid from~$(\uu, \vv)$, and for all~$\uu', \vv'$ satisfying $\uu' \eqR \uu$ and $\vv' \eqR \vv$, \\ \null\hfill there is an equivalent grid from~$(\uu', \vv')$.}
\end{equation}
\end{lemm}

\begin{proof}
Assume that \eqref{E:Compat} holds for all~$\uu, \vv$ in~$\SSSs$, and let $\uu, \uu'$ be $\eqR$-equivalent words. A trivial induction on the length of~$\uu$ shows 

\noindent\begin{minipage}{\textwidth}
\rightskip48mm that there exists \VR(3.2,0)a grid~$\Gamma$ from~$(\uu, \uu)$ to~$(\ew, \ew)$, as shown on the right. Applying~\eqref{E:Compat} to~$\Gamma$ and to the equivalences $\uu \eqR \uu$ and $\uu \eqR \uu'$, we conclude that there exists a grid~$\Gamma'$ from~$(\uu, \uu')$ that is equivalent to~$\Gamma$. Let $(\uu'_1, \vv'_1)$ be the output of~$\Gamma'$. Then, by assumption, we have $\uu'_1 \eqR \ew$ and $\vv'_1 \eqR \ew$. Because $\RRR$ contains no $\ew$-relation, $\uu'_1 \eqR \ew$ implies that $\uu'_1$ is empty, and $\vv'_1 \eqR \ew$ implies that $\vv'_1 $ is empty.\hfill
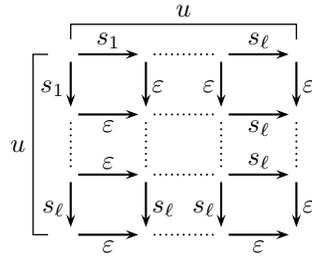
\begin{picture}(0,0)(-13,-4)
\pcline{->}(1,24)(9,24)\taput{$\ss_1$}
\pcline[style=etc](11,24)(19,24)
\pcline{->}(21,24)(29,24)\taput{$\ss_\ell$}
\pcline{->}(1,16)(9,16)\tbput{$\ew$}
\pcline[style=etc](11,16)(19,16)
\pcline{->}(21,16)(29,16)\tbput{$\ss_\ell$}
\pcline{->}(1,8)(9,8)\taput{$\ew$}
\pcline[style=etc](11,8)(19,8)
\pcline{->}(21,8)(29,8)\taput{$\ss_\ell$}
\pcline{->}(1,0)(9,0)\tbput{$\ew$}
\pcline[style=etc](11,0)(19,0)
\pcline{->}(21,0)(29,0)\tbput{$\ew$}
\pcline{->}(0,23)(0,17)\tlput{$\ss_1$}
\pcline[style=etc](0,15)(0,9)
\pcline{->}(0,7)(0,1)\tlput{$\ss_\ell$}
\pcline{->}(10,23)(10,17)\trput{$\ew$}
\pcline[style=etc](10,15)(10,9)
\pcline{->}(10,7)(10,1)\trput{$\ss_\ell$}
\pcline{->}(20,23)(20,17)\tlput{$\ew$}
\pcline[style=etc](20,15)(20,9)
\pcline{->}(20,7)(20,1)\tlput{$\ss_\ell$}
\pcline{->}(30,23)(30,17)\trput{$\ew$}
\pcline[style=etc](30,15)(30,9)
\pcline{->}(30,7)(30,1)\trput{$\ew$}
\psline[linewidth=0.5pt](-3,24)(-5,24)(-5,0)(-3,0)\put(-8,11){$\uu$}
\psline[linewidth=0.5pt](0,26)(0,28)(30,28)(30,26)\put(14,29){$\uu$}
\end{picture}
\end{minipage}
So \VR(3.2,0) $\Gamma'$ witnesses that $(\uu, \uu')$ right-reverses to~$(\ew, \ew)$ and, therefore, right reversing is complete for~$(\SSS, \RRR)$.
\end{proof}

As it stands, Lemma~\ref{L:CompatCompl} does not provide a tractable criterion, because it involves arbitrary pairs of $\eqR$-equivalent words in~$\SSSs$. We show now that, under convenient finiteness assumptions (``noetherianity''), the most elementary instances of the condition are sufficient to deduce the full condition.

If $\MM$ is a monoid and $\gg, \hh$ belong to~$\MM$, one says that $\gg$ \emph{properly right-divides}~$\hh$, written $\gg \divR \hh$ or $\hh \multR \gg$, if $\hh = \hh' \gg$ holds for some non-invertible element~$\hh'$ of~$\MM$ (proper left-division~$\divL$ would be defined symmetrically with~$\gg$ on the left).

\begin{defi}
A monoid~$\MM$ is called \emph{right noetherian} if there is no infinite descending sequence with respect to proper right-divisibili\-ty relation in~$\MM$, that is, every sequence $\gg_0 \multR\nobreak \gg_1 \multR\nobreak ...$ in~$\MM$ is finite\footnote{ This notion of noetherianity, which is reminiscent of that of rings and algebras, is not the only one used for semigroups: right noetherianity may also refer to a monoid in which every right congruence is finitely generated, or to a monoid in which every right ideal is finitely generated.}. 
\end{defi}

 In a general monoid, the notions of left invertible, right invertible, and invertible elements need not coincide. That difficulty vanishes in a right noetherian monoid.

\begin{lemm}
Assume that $\MM$ is a right noetherian monoid.

\ITEM1 An element of~$\MM$ is left invertible if, and only if, it is right invertible if, and only if, it is invertible.

\ITEM2 The product of two non-invertible elements of~$\MM$ is non-invertible.
\end{lemm}

\begin{proof}
\ITEM1 First, we recall that, if an element admits a left and a right inverse, then the latter are equal, for $\ff \gg = \gg \ff' = 1$ implies $\ff = \ff (\gg\ff') = (\ff\gg) \ff' = \ff'$.

Now, assume that $\gg$ admits a left inverse, say $\ff \gg = 1$. Two cases are possible. If $\ff$ is invertible, then so is~$\gg$, since $\hh \ff = 1$ implies $\hh = \hh \ff \gg = \gg$, whence $\gg \ff =\nobreak 1$. Otherwise, $\gg^\kk = \ff \gg \gg^\kk$ gives $\gg^\kk \multR \gg^{\kk + 1}$ for every~$\kk \ge 0$, leading to the infinite descending sequence $1 \multR \gg \multR \gg^2 \multR ...$, which contradicts right noetherianity. Hence left invertibility implies invertibility in~$\MM$.

Next, assume that $\gg$ admits a right inverse, say $\gg \hh = 1$. Then $\hh$ admits a left inverse and, by the above result, $\hh$ must be invertible. This in turn implies that $\hh$ is also a left inverse of~$\gg$, so $\gg$ is invertible, and right invertibility implies invertibility.

\ITEM2 Assume that $\gg$ and~$\hh$ are non-invertible elements of~$\MM$, and $\gg\hh$ is invertible. Then $\gg$ is right invertible and $\hh$ is left invertible, so, by~\ITEM1, both are invertible, which implies that their product is invertible, a contradiction.
\end{proof}

Recognizing whether a monoid is noetherian is in general difficult. In practice, we can use the following criterion.

\begin{lemm}\label{L:NonInv}
For every monoid~$\MM$, the following are equivalent:

\ITEM1 The monoid~$\MM$ is right noetherian.

\ITEM2 There exists a map~$\lambda$ from~$\MM$ to ordinals such that, for all~$\gg, \gg'$ in~$\MM$, 
\begin{equation}\label{E:Wit00}
\text{$\gg \multR \gg'$ implies $\lambda(\gg) > \lambda(\gg')$}. 
\end{equation}

\ITEM3 There exists a map~$\lambda$ from~$\MM$ to ordinals satisfying, for all~$\gg, \hh$ in~$\MM$, 
\begin{equation}\label{E:Wit0}
\text{$\lambda(\gg\hh) \ge \lambda(\hh) + \lambda(\gg)$, \ and \ $\lambda(\gg) > 0$ whenever $\gg$ is non-invertible}. 
\end{equation}
\end{lemm}

\begin{proof}
The equivalence of~\ITEM1 and~\ITEM2 is standard: for~$\multR$ to admit no infinite descending sequence means that the relation~$\multR$ is well-founded, and it is well known that this amounts to the existence of a map to the ordinals that decreases along~$\multR$.

Next, \ITEM3 implies~\ITEM2: indeed, assuming $\gg = \hh\gg'$ with~$\hh$ non-invertible and applying~\eqref{E:Wit0}, we obtain $\lambda(\gg) \ge \lambda(\gg') + \lambda(\hh) > \lambda(\gg')$.

Finally, assume~\ITEM1, whence~\ITEM2. As above, the relation~$\multR$ is well founded, so, by standard arguments, there exists a map $\lambda: \MM \to \Ord$ inductively defined by
\begin{equation}\label{E:Wit1}
\lambda(\gg) := 
\begin{cases}
0 &\text{if $\gg$ is invertible},\\
\sup\{\lambda(\ff) + 1 \mid \ff \divR \gg\} &\text{otherwise}.
\end{cases}
\end{equation}
We claim that this particular function~$\lambda$, which satisfies~\eqref{E:Wit00} by construction, also satisfies~\eqref{E:Wit0}. 
First, we observe that, if $\gg$ is not invertible, then $\gg \multR 1$ is true, so we must have $\lambda(\gg) > \lambda(1) = 0$. So the second assertion in~\eqref{E:Wit0} is true. Next, we observe that, if $\gg$ is invertible, then $\lambda(\gg\hh) \ge \lambda(\hh)$ holds. Indeed, the inequality is trivial for~$\lambda(\hh) = 0$, and, otherwise, the sets $\{\ff \mid \ff \divR \hh\}$ and~$\{\ff \mid \ff \divR \gg\hh\}$ coincide, and we deduce 
$$\lambda(\gg\hh) = \sup\{\lambda(\ff) + 1 \mid \ff \divR \gg\hh\} = \sup\{\lambda(\ff) + 1 \mid \ff \divR \hh\} = \lambda(\hh).$$
We prove now using induction on~$\lambda(\gg)$ that $\lambda(\gg\hh) \ge \lambda(\hh) + \lambda(\gg)$ holds for every~$\hh$ in~$\MM$. Assume first $\lambda(\gg) = 0$. Then $\gg$ must be invertible, and we established above the equality $\lambda(\gg\hh) = \lambda(\hh) = \lambda(\hh) + \lambda(\gg)$, as expected. Assume now $\lambda(\gg) > 0$. Then $\gg$ is not invertible and, by definition, we have $\lambda(\gg) = \sup\{\lambda(\ff) + 1 \mid \ff \divR \gg\}$. Let~$\hh$ be an arbitrary element of~$\MM$. By Lemma~\ref{L:NonInv}, $\gg\hh$ is not invertible, and we obtain
\begin{align*}
\lambda(\gg\hh)
&= \sup\{\lambda(\ff) + 1 \mid \ff \divR \gg\hh\}
&&\text{by definition}\\
&\ge \sup\{\lambda(\ff \hh) + 1 \mid \ff \divR \gg\}
&&\text{because $\ff \divR \gg$ implies $\ff \hh \divR \gg\hh$}\\
&\ge \sup\{\lambda(\hh) + \lambda(\ff) + 1 \mid \ff \divR \gg\}
&&\text{by induction hypothesis}\\
&= \lambda(\hh) + \sup\{\lambda(\ff) + 1 \mid \ff \divR \gg\}
&&\text{by monotonicity of ordinal addition}\\
&= \lambda(\hh) + \lambda(\gg)
&&\text{by definition}.
\end{align*}
Thus the first inequality in~\eqref{E:Wit0} is established, and \ITEM1 implies~\ITEM3.
\end{proof}

Translating the previous result at the level of presentations, we can state:

\begin{lemm}\label{L:Noeth}
If $(\SSS, \RRR)$ is a monoid presentation, the mon\-oid~$\MON\SSS\RRR$ is right noetherian whenever the following equivalent conditions hold:
\begin{gather}\label{E:Wit1}
\parbox{112mm}{%
there exists an $\eqR$-invariant map~$\lambda$ from~$\SSSs$ to the ordinals\\
\null\hfill \VR(3.2,0)satisfying $\lambda(\ss\ww) > \lambda(\ww)$ for all~$\ss$ in~$\SSS$ and~$\ww$ in~$\SSSs$;}\\
\label{E:Wit11}
\parbox{112mm}{%
there exists an $\eqR$-invariant map~$\lambda$ from~$\SSSs$ to the ordinals\\
\null\hfill \VR(3.2,0)satisfying $\lambda(\uu\vv) \ge \lambda(\vv) + \lambda(\uu)$ for all $\uu, \vv$ in~$\SSSs$, and $\lambda(\ss) > 0$ for~$\ss$ in~$\SSS$.} 
\end{gather}
\end{lemm}

Thus, \eqref{E:Wit1} provides a sufficient condition for right noetherianity---which is also necessary if no element of~$\SSS$ is invertible in~$\MON\SSS\RRR$---and, when it is satisfied, one is assured that the stronger condition~\eqref{E:Wit11} is satisfied (possibly by another map~$\lambda'$). As already noted, in the case of a homogeneous presentation, \ie, when all relations have the form $\ww \,{=}\, \ww'$ with $\ww, \ww'$ of the same length, defining~$\lambda(\ww)$ to be the length of~$\ww$ provides a map~$\lambda$ witnessing~\eqref{E:Wit11}. Note that \eqref{E:Wit1} can hold only if there is no $\ew$-relation so, when considering below monoid presentations that satisfy~\eqref{E:Wit1}, there is no need to explicitly require that they contain no $\ew$-relation.

The main technical result we shall establish is the following criterion for the completeness of reversing:

\begin{lemm}\label{L:Main}
A monoid presentation~$(\SSS, \RRR)$ satisfying~\eqref{E:Wit1} satisfies~\eqref{E:Compat} if, and only if, for every element~$\ss$ in~$\SSS$ and every relation $\ww \,{=}\, \ww'$ in~$\RRR$, 
\begin{equation}\label{E:CompatS}\tag{$\Cond$}
\parbox{113mm}{for every grid from~$(\ss, \ww)$, there is an equivalent grid from~$(\ss, \ww')$, \\ \null\hfill and vice versa.}
\end{equation}
\end{lemm}

 The proof will use an induction on an ordinal parameter called the diagonal of a grid:

\begin{defi}
If $(\SSS, \RRR)$ is a monoid presentation and $\lambda$ is a map witnessing~\eqref{E:Wit11}, then, if $\Gamma$ is an $(\SSS, \RRR)$-grid from~$(\uu, \vv)$ to~$(\uu_1, \vv_1)$, the \emph{diagonal} of~$\Gamma$ is the value of~$\lambda(\uu\vv_1)$. 
\end{defi}

Note that, with the above notation and by Lemma~\ref{L:RevEq}, the diagonal of~$\Gamma$ is also equal to~$\lambda(\vv\uu_1)$. 

On the other hand, for~$\ww, \ww'$ in~$\SSSs$, we write $\distR(\ww, \ww')$ for the \emph{combinatorial distance} between~$\ww$ and~$\ww'$ with respect to~$\RRR$, namely the minimal length of an $\RRR$-derivation from~$\ww$ to~$\ww'$ if $\ww$ and~$\ww'$ are $\eqR$-equivalent, and~$\infty$ otherwise. 

\begin{proof}[Proof of Lemma~\ref{L:Main}]
One implication is trivial: \eqref{E:CompatS} for~$\ss$ and~$\ww \,{=}\, \ww'$ follows from applying~\eqref{E:Compat} to the words~$\ss$ and~$\ww$ with the equivalences $\ss \eqR \ss$ and~$\ww' \eqR \ww$.

The point is to establish the converse implication. This will be done using two nested inductions. First, we fix a map~$\lambda$ from~$\SSSs$ to ordinals satisfying~\eqref{E:Wit11}, which is possible by Lemma~\ref{L:Noeth}. By the properties of ordinal addition, we always have
\begin{equation}\label{E:Ineq}
\lambda(\uu) \le \lambda(\uu\vv), \quad \lambda(\vv) \le \lambda(\uu\vv), \text{\quad and\quad $\lambda(\vv) < \lambda(\uu\vv)$ for $\uu$ nonempty}.
\end{equation}
Then, for~$\alpha$ an ordinal, we introduce the special case of Condition~\eqref{E:Compat} corresponding to grids whose diagonal is at most~$\alpha$:
\begin{equation}\label{E:Compat1}\tag{$\Conds\alpha$}
\parbox{113mm}{For every grid with diagonal~${\le}\,\alpha$ from~$(\uu, \vv)$, and for all~$\uu', \vv'$ satisfying \\ \null\hfill $\uu' \eqR \uu$ and $\vv' \eqR \vv$, there is an equivalent grid from~$(\uu', \vv')$.}
\end{equation}
 Finally, for~$\dd$ a natural number, we consider the special case of Condition~\eqref{E:Compat1} corresponding to combinatorial distances between the sources of the old and new grids bounded by~$\dd$:
\begin{equation}\label{E:Compat11}\tag{$\Conds{\alpha, \dd}$}
\parbox{113mm}{For every grid with diagonal~${\le}\,\alpha$ from~$(\uu, \vv)$, and for all~$\uu', \vv'$ satisfying \\ \null\hfill $\distR(\uu, \uu') + \distR(\vv, \vv') \le \dd$, there is an equivalent grid from~$(\uu', \vv')$.}
\end{equation}
It should be clear that \eqref{E:Compat} for two words~$\uu, \vv$ is equivalent to the conjunction of all~\eqref{E:Compat11} for~$\uu, \vv$. Using an induction on~$\alpha$ and, for a given~$\alpha$, on~$\dd$, we shall establish that, if \eqref{E:CompatS} is true for every~$\ss$ and every relation~$\ww \,{=}\, \ww'$ of~$\RRR$, then \eqref{E:Compat11} is true for all~$\uu, \vv$.

Assume first $\alpha = 0$. Assume that $\Gamma$ is a grid with zero diagonal from~$(\uu, \vv)$ to~$(\uu_1, \vv_1)$, and $\uu' \eqR \uu$ and $\vv' \eqR \vv$ hold. By construction, $\lambda(\ww) = 0$ implies that $\ww$ is empty, so $\lambda(\uu\vv_1) = \lambda(\vv \uu_1) = \nobreak 0$ requires that $\uu, \vv, \uu_1$, and~$\vv_1$ all are empty. Next, the assumption $\uu' \eqR \ew$ implies that $\uu'$ is empty, and $\vv' \eqR \vv$ implies that $\vv'$ is empty as well. Then choosing $\Gamma' := \Gamma$ provides the expected condition. So $(\Conds0)$ is true for all~$\uu, \vv$.

Assume now $\alpha > 0$ and $\dd = 0$. Assume that $\Gamma$ is a grid with diagonal~${\le}\,\alpha$ from~$(\uu, \vv)$ to~$(\uu_1, \vv_1)$, and $\distR(\uu, \uu') + \distR(\vv, \vv') = 0$ holds. By definition, we have $\uu' \weq \uu$ and $\vv' \weq \vv$. Then choosing $\Gamma' := \Gamma$ provides the expected condition. So $(\Conds{\alpha,0})$ is true for all~$\uu, \vv$ and for every~$\alpha$.

Assume now $\alpha > 0$ and $\dd = 1$. Assume that $\Gamma$ is a grid with diagonal~${\le}\,\alpha$ from~$(\uu, \vv)$ to~$(\uu_1, \vv_1)$, and $\distR(\uu, \uu') + \distR(\vv, \vv') = 1$ holds. Up to a symmetry, we may assume $\uu' \weq \uu$ and $\distR(\vv', \vv) = 1$. By definition, the latter relation means that there exists a relation~$\ww \,{=}\, \ww'$ in~$\RRR$ and two words~$\vv_0, \vv_2$ satisfying $\vv \weq \vv_0 \ww \vv_2$ and $\vv' \weq \vv_0 \ww' \vv_2$. As $\vv$ is the product $\vv_0 \ww \vv_2$, repeated applications of Lemma~\ref{L:Decomp} show that the assumption $(\uu, \vv) \revR (\uu_1, \vv_1)$ implies the existence of~$\uu_0$, $\uu_2$ and $\vv_3$, $\vv_4$, and~$\vv_5$ satisfying $\vv_1 \weq \vv_3\vv_4\vv_5$ and
$$(\uu, \vv_0) \revR (\uu_0, \vv_3), \quad (\uu_0, \ww) \revR (\uu_2, \vv_4), \quad \text{and} \quad (\uu_2, \vv_2) \revR (\uu_1, \vv_5),$$
corresponding to a decomposition of the grid~$\Gamma$ into the union of three grids
$$
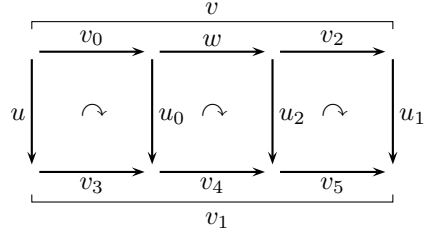
\begin{picture}(50,28)(0,-5)
\pcline{->}(1,16)(15,16)\taput{$\vv_0$}
\pcline{->}(17,16)(31,16)\taput{$\ww$}
\pcline{->}(33,16)(47,16)\taput{$\vv_2$}
\pcline{->}(1,0)(15,0)\tbput{$\vv_3$}
\pcline{->}(17,0)(31,0)\tbput{$\vv_4$}
\pcline{->}(33,0)(47,0)\tbput{$\vv_5$}
\psline[linewidth=0.4pt](0,-3)(0,-4)(48,-4)(48,-3)\put(23,-7){$\vv_1$}
\psline[linewidth=0.4pt](0,19)(0,20)(48,20)(48,19)\put(23,21){$\vv$}
\pcline{->}(0,15)(0,1)\tlput{$\uu$}
\pcline{->}(16,15)(16,1)\trput{$\uu_0$}
\pcline{->}(32,15)(32,1)\trput{$\uu_2$}
\pcline{->}(48,15)(48,1)\trput{$\uu_1$}
\put(6.5,7.5){$\rev$}
\put(22.5,7.5){$\rev$}
\put(38.5,7.5){$\rev$}
\end{picture}$$
Assume first that the word~$\uu_0$ is empty. Then, necessarily, $\uu_1$ and $\uu_2$ are empty, and we have $\vv_4 \weq \ww$ and $\vv_5 \weq \vv_2$. Then the situation is as the left diagram below
$$
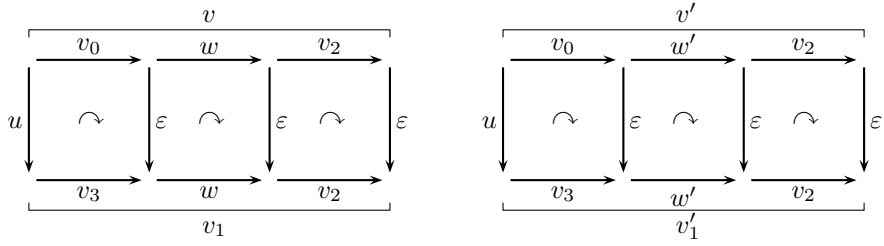
\begin{picture}(50,30)(0,-6)
\pcline{->}(1,16)(15,16)\taput{$\vv_0$}
\pcline{->}(17,16)(31,16)\taput{$\ww$}
\pcline{->}(33,16)(47,16)\taput{$\vv_2$}
\pcline{->}(1,0)(15,0)\tbput{$\vv_3$}
\pcline{->}(17,0)(31,0)\tbput{$\ww$}
\pcline{->}(33,0)(47,0)\tbput{$\vv_2$}
\psline[linewidth=0.4pt](0,-3)(0,-4)(48,-4)(48,-3)\put(23,-7){$\vv_1$}
\psline[linewidth=0.4pt](0,19)(0,20)(48,20)(48,19)\put(23,21){$\vv$}
\pcline{->}(0,15)(0,1)\tlput{$\uu$}
\pcline{->}(16,15)(16,1)\trput{$\ew$}
\pcline{->}(32,15)(32,1)\trput{$\ew$}
\pcline{->}(48,15)(48,1)\trput{$\ew$}
\put(6.5,7.5){$\rev$}
\put(22.5,7.5){$\rev$}
\put(38.5,7.5){$\rev$}
\end{picture}\hspace{13mm}
\begin{picture}(50,30)(0,-6)
\pcline{->}(1,16)(15,16)\taput{$\vv_0$}
\pcline{->}(17,16)(31,16)\taput{$\ww'$}
\pcline{->}(33,16)(47,16)\taput{$\vv_2$}
\pcline{->}(1,0)(15,0)\tbput{$\vv_3$}
\pcline{->}(17,0)(31,0)\tbput{$\ww'$}
\pcline{->}(33,0)(47,0)\tbput{$\vv_2$}
\psline[linewidth=0.4pt](0,-3)(0,-4)(48,-4)(48,-3)\put(23,-7){$\vv'_1$}
\psline[linewidth=0.4pt](0,19)(0,20)(48,20)(48,19)\put(23,21){$\vv'$}
\pcline{->}(0,15)(0,1)\tlput{$\uu$}
\pcline{->}(16,15)(16,1)\trput{$\ew$}
\pcline{->}(32,15)(32,1)\trput{$\ew$}
\pcline{->}(48,15)(48,1)\trput{$\ew$}
\put(6.5,7.5){$\rev$}
\put(22.5,7.5){$\rev$}
\put(38.5,7.5){$\rev$}
\end{picture}$$
and the right diagram shows that $(\Conds{\alpha, 1})$ is satisfied with $\uu'_1$ empty and $\vv'_1 \weq \nobreak \vv_3\ww'\vv_2$.

Assume now that $\uu_0$ is not empty. Then we write $\uu_0 \weq \ss \uu_3$ with~$\ss$ in~$\SSS$. Splitting the grid again, we obtain the existence of words $\uu_4 \wdots \uu_7$ and $\vv_6, \vv_7$ such that the situation is as in the left diagram below
$$
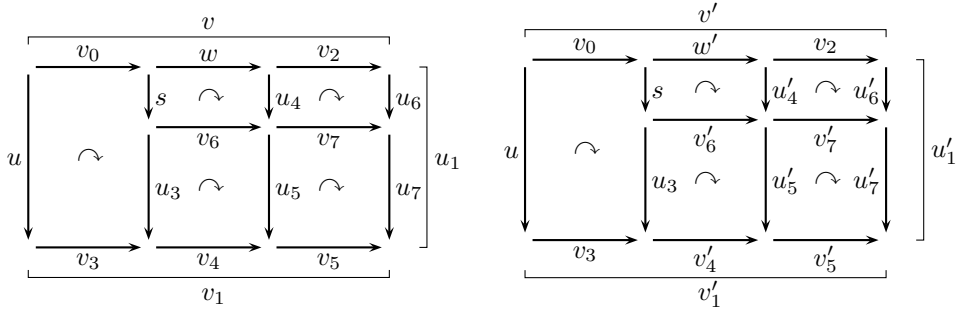
\begin{picture}(53,38)(0,-6)
\pcline{->}(1,24)(15,24)\taput{$\vv_0$}
\pcline{->}(17,24)(31,24)\taput{$\ww$}
\pcline{->}(33,24)(47,24)\taput{$\vv_2$}
\pcline{->}(17,16)(31,16)\tbput{$\vv_6$}
\pcline{->}(33,16)(47,16)\tbput{$\vv_7$}
\pcline{->}(1,0)(15,0)\tbput{$\vv_3$}
\pcline{->}(17,0)(31,0)\tbput{$\vv_4$}
\pcline{->}(33,0)(47,0)\tbput{$\vv_5$}
\psline[linewidth=0.4pt](0,-3)(0,-4)(48,-4)(48,-3)\put(23,-7){$\vv_1$}
\psline[linewidth=0.4pt](0,27)(0,28)(48,28)(48,27)\put(23,29){$\vv$}
\pcline{->}(0,23)(0,1)\tlput{$\uu$}
\pcline{->}(16,23)(16,17)\trput{$\ss$}
\pcline{->}(16,15)(16,1)\trput{$\uu_3$}
\pcline{->}(32,23)(32,17)\trput{$\uu_4$}
\pcline{->}(32,15)(32,1)\trput{$\uu_5$}
\pcline{->}(48,23)(48,17)\trput{$\uu_6$}
\pcline{->}(48,15)(48,1)\trput{$\uu_7$}
\psline[linewidth=0.4pt](52,24)(53,24)(53,0)(52,0)\put(54,11){$\uu_1$}
\put(6.5,11.5){$\rev$}
\put(22.5,7.5){$\rev$}
\put(38.5,7.5){$\rev$}
\put(22.5,19.5){$\rev$}
\put(38.5,19.5){$\rev$}
\end{picture}\hspace{13mm}
\begin{picture}(53,35)(0,-7)
\pcline{->}(1,24)(15,24)\taput{$\vv_0$}
\pcline{->}(17,24)(31,24)\taput{$\ww'$}
\pcline{->}(33,24)(47,24)\taput{$\vv_2$}
\pcline{->}(17,16)(31,16)\tbput{$\vv_6'$}
\pcline{->}(33,16)(47,16)\tbput{$\vv_7'$}
\pcline{->}(1,0)(15,0)\tbput{$\vv_3$}
\pcline{->}(17,0)(31,0)\tbput{$\vv_4'$}
\pcline{->}(33,0)(47,0)\tbput{$\vv_5'$}
\psline[linewidth=0.4pt](0,-4)(0,-5)(48,-5)(48,-4)\put(23,-8){$\vv'_1$}
\psline[linewidth=0.4pt](0,27)(0,28)(48,28)(48,27)\put(23,29){$\vv'$}
\pcline{->}(0,23)(0,1)\tlput{$\uu$}
\pcline{->}(16,23)(16,17)\trput{$\ss$}
\pcline{->}(16,15)(16,1)\trput{$\uu_3$}
\pcline{->}(32,23)(32,17)\trput{$\uu_4'$}
\pcline{->}(32,15)(32,1)\trput{$\uu_5'$}
\pcline{->}(48,23)(48,17)\tlput{$\uu_6'$}
\pcline{->}(48,15)(48,1)\tlput{$\uu_7'$}
\psline[linewidth=0.4pt](52,24)(53,24)(53,0)(52,0)\put(54,11){$\uu'_1$}
\put(6.5,11.5){$\rev$}
\put(22.5,7.5){$\rev$}
\put(38.5,7.5){$\rev$}
\put(22.5,19.5){$\rev$}
\put(38.5,19.5){$\rev$}
\end{picture}$$
We shall now establish the existence of words~$\uu'_3 \wdots \uu'_7$ and $\vv'_4 \wdots \vv'_7$ such that the right diagram above is a legitimate $(\SSS, \RRR)$-grid, with $\uu'_\ii \eqR \uu_\ii$ and $\vv'_\jj \eqR \vv_\jj$ for all~$\ii$ and~$\jj$.

We begin with the top median square. By assumption, we have $(\ss, \ww) \revR (\uu_4, \vv_6)$ and $\ww \,{=}\, \ww' \in~\RRR$. By~\eqref{E:CompatS}, there exist~$\uu'_4$ and~$\vv'_6$ satisfying
$$\uu'_4 \eqR \uu_4, \quad \vv'_6\eqR \vv_6 \quad \text{and }\quad (\ss, \ww') \rev_\RRR (\uu'_4, \vv'_6).$$ 

Consider now the bottom median square. Then $\uu_3 \eqR \uu_3$ is trivial, whereas $\vv'_6 \eqR \vv_6$ and $(\uu_3, \vv_6) \revR (\uu_5, \vv_4)$ hold by construction. Moreover, \eqref{E:Ineq} implies 
$$\lambda(\uu_3\vv_4) < \lambda(\ss\uu_3\vv_4) \le \lambda(\vv_0\ss\uu_3\vv_4) \le \lambda(\vv_0\ss\uu_3\vv_4 \vv_5) = \lambda(\uu\vv_1) \le \alpha,$$ 
whence $\beta:= \lambda(\uu_3\vv_4) < \alpha$. By induction hypothesis, $(\Conds\beta)$ is true for~$\uu_3$ and~$\vv_6$, and we deduce the existence of~$\uu'_5$ and~$\vv'_4$ satisfying 
$$\uu'_5 \eqR \uu_5, \quad \vv'_4\eqR \vv_4 \quad \text{and }\quad (\uu_3, \vv'_6) \rev_\RRR (\uu'_5, \vv'_4).$$ 

We move to the top right square. Then $\vv_2 \eqR \vv_2$ is trivial, whereas $\uu'_4 \eqR \uu_4$ and $(\uu_4, \vv_2) \revR (\uu_6, \vv_7)$ hold by construction. Moreover, because $\ww$ cannot be empty, since $\RRR$ contains no $\ew$-relation, \eqref{E:Ineq} implies 
$$\lambda(\uu_4\vv_7) < \lambda(\ww\uu_4\vv_7) \le \lambda(\vv_0\ww\uu_4\vv_7) \le \lambda(\vv_0\ww\uu_4\vv_7\uu_7) = \lambda(\uu\vv_1) \le \alpha,$$ 
whence $\gamma:= \lambda(\uu_4\vv_7) < \alpha$. By induction hypothesis, $(\Conds\gamma)$ is true for~$\uu_4$ and~$\vv_2$, and we deduce the existence of~$\uu'_6$ and~$\vv'_7$ satisfying 
$$\uu'_6 \eqR \uu_6, \quad \vv'_7\eqR \vv_7 \quad \text{and }\quad (\uu'_4, \vv_2) \rev_\RRR (\uu'_6, \vv'_7).$$

Finally, we consider the bottom right square. By construction, we have $\uu'_5 \eqR\nobreak \uu_5$, $\vv'_7 \eqR \vv_7$ and $(\uu_5, \vv_7) \revR (\uu_7, \vv_5)$. Moreover, \eqref{E:Ineq} implies 
$$\lambda(\uu_5\vv_5) \le \lambda(\vv_6\uu_5\vv_5) < \lambda(\ss\vv_6\uu_5\vv_5) \le \lambda(\vv_0\ss\vv_6\uu_5\vv_5) = \lambda(\uu\vv_1) \le \alpha,$$ 
whence $\delta:= \lambda(\uu_5\vv_5) < \alpha$. By induction hypothesis, $(\Conds\delta)$ is true for~$\uu_5$ and~$\vv_7$, and we deduce the existence of~$\uu'_7$ and~$\vv'_5$ satisfying 
$$\uu'_7 \eqR \uu_7, \quad \vv'_5\eqR \vv_5 \quad \text{and }\quad (\uu'_5, \vv'_7) \rev_\RRR (\uu'_7, \vv'_5).$$

Put $\uu'_1 \weq \uu'_6\uu'_7$ and $\vv'_1 \weq \vv_3\vv'_4\vv'_5$. Then $\uu'_1 \eqR \uu_1$ and $\vv'_1 \eqR \vv_1$ hold, and the right diagram below witnesses $(\vv, \vv') \revR (\uu'_1, \vv'_1)$. Thus $(\Conds{\alpha, 1})$ is satisfied for~$\uu$ and~$\vv$, which completes the case $\dd = 1$ in the induction for~$(\Conds\alpha)$.

Assume finally $\alpha > 0$ and $\dd \ge 2$. Assume that $\Gamma$ is a grid with diagonal~${\le}\,\alpha$ from~$(\uu, \vv)$ to~$(\uu_1, \vv_1)$, and $\distR(\uu, \uu') + \distR(\vv, \vv') = \dd$ holds. We can find two words~$\uu'', \vv''$ satisfying 
$$\distR(\uu'', \uu) + \distR(\vv'', \vv) = \dd -1 \quad \text{and}\quad\distR(\uu', \uu'') + \distR(\vv', \vv'') = 1.$$
By assumption, we have $\lambda(\uu\vv_1) \le \alpha$. By induction hypothesis, $(\Conds{\alpha, \dd - 1})$ is true for~$\uu$ and~$\vv$, so we deduce the existence of~$\uu''_1, \vv''_1$ satisfying
$$\uu''_1 \eqR \uu_1, \quad \vv''_1 \eqR \vv_1, \quad \text{and} \quad (\uu'', \vv'') \revR (\uu''_1, \vv''_1).$$
Now $\uu'' \eqR \uu$ and $\vv''_1 \eqR \vv_1$ imply $\lambda(\uu''\vv''_1) = \lambda(\uu\vv_1) \le \alpha$. By induction hypothesis, $(\Conds{\alpha, 1})$ is true for~$\uu''$ and~$\vv''$, so we deduce the existence of~$\uu'_1, \vv'_1$ satisfying
$$\uu'_1 \eqR \uu''_1, \quad \vv'_1 \eqR \vv''_1, \quad and \quad (\uu', \vv') \revR (\uu'_1, \vv'_1).$$
By transitivity of~$\eqR$, we have $\uu'_1 \eqR \uu_1$ and $\vv'_1 \eqR \vv_1$, and we conclude that $(\Conds{\alpha, \dd})$ is true for~$\uu$ and~$\vv$. This completes the induction.
\end{proof}

\subsection{Main results}\label{SS:Cancel}

We are now ready to state the main results of the paper and, in particular, to establish the cancellativity criterion announced in the title.

First, summarizing the results established so far directly gives the following:

\begin{prop}
Assume that a monoid presentation~$(\SSS, \RRR)$ satisfies~\eqref{E:Wit1} and~\eqref{E:CompatS} for every~$\ss$ in~$\SSS$ and every relation~$\ww \,{=}\, \ww'$ in~$\RRR$.

\ITEM1 For all~$\uu, \vv, \uu_1, \vv_1$ in~$\SSSs$ satisfying $(\uu, \vv) \rev_\RRR (\uu_1, \vv_1)$, and for all~$\uu', \vv'$ in~$\SSSs$ satisfying $\uu' \eqR \uu$ and $\vv' \eqR \vv$, there exist~$\uu'_1, \vv'_1$ satisfying $(\uu', \vv') \rev_\RRR (\uu'_1, \vv'_1)$, with $\uu'_1 \eqR \uu_1$ and $\vv'_1 \eqR \vv_1$.

\ITEM2 For all~$\uu, \vv$ in~$\SSSs$, the words $\uu$ and~$\vv$ represent the same element of the monoid~$\MON\SSS\RRR$ if, and only if, $(\uu, \vv) \rev_\RRR (\ew, \ew)$ holds.
\end{prop}

\begin{proof}
Point~\ITEM1 is Condition~\eqref{E:Compat} for~$\uu, \vv$, and Lemma~\ref{L:Main} states that the latter holds whenever \eqref{E:Wit1} holds and so does~\eqref{E:CompatS} for every~$\ss$ in~$\SSS$ and every relation~$\ww \,{=}\, \ww'$ in~$\RRR$.

\ITEM2 By Lemma~\ref{L:CompatCompl}, \ITEM1, that is, \eqref{E:Compat} for all~$\uu, \vv$, implies that reversing is complete for~$(\SSS, \RRR)$, which, by definition, implies the equivalence of~\ITEM2. 
\end{proof}

Let us turn to left cancellativity. Then completeness of right reversing is useful, as it shows that, if there is no obvious counter-example to left cancellativity, then there is no hidden counter-example either:

\begin{lemm}
If right reversing is complete for the presenta\-tion~$(\SSS, \RRR)$ and $\RRR$ contains no relation of the form $\ss\uu \,{=}\, \ss\vv$ with $\ss$ in~$\SSS$ and $\uu, \vv$ distinct, then the monoid~$\MON\SSS\RRR$ admits left cancellation.
\end{lemm}

\begin{proof}
It is enough to prove that, for all words~$\uu, \vv$ in~$\SSSs$, every relation of the form $\ss\uu \eqR \ss\vv$ with~$\ss$ in~$\SSS$ implies $\uu \eqR \vv$. So \VR(3.5,0) assume $\ss\uu \eqR \ss\vv$. 

\noindent\begin{minipage}{\textwidth}
\rightskip30mm By \VR(3.5,0) completeness of right reversing, we deduce $(\ss\uu, \ss\vv) \revR (\ew, \ew)$. By Lemma~\ref{L:Decomp}, a grid witnessing $(\ss\uu, \ss\vv) \revR (\ew, \ew)$ decomposes into four grids, as shown on the right. The assumption about the presentation requires that $\uu_1$ and~$\vv_1$ are empty, which in turn implies $\uu_2 \weq \uu$ and $\vv_2 \weq \vv$. \hfill
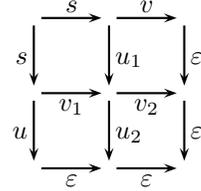
\begin{picture}(0,0)(-7,-2)
\pcline{->}(1,20)(9,20)\taput{$\ss$}
\pcline{->}(11,20)(19,20)\taput{$\vv$}
\pcline{->}(1,10)(9,10)\tbput{$\vv_1$}
\pcline{->}(11,10)(19,10)\tbput{$\vv_2$}
\pcline{->}(1,0)(9,0)\tbput{$\ew$}
\pcline{->}(11,0)(19,0)\tbput{$\ew$}
\pcline{->}(0,19)(0,11)\tlput{$\ss$}
\pcline{->}(0,9)(0,1)\tlput{$\uu$}
\pcline{->}(10,19)(10,11)\trput{$\uu_1$}
\pcline{->}(10,9)(10,1)\trput{$\uu_2$}
\pcline{->}(20,19)(20,11)\trput{$\ew$}
\pcline{->}(20,9)(20,1)\trput{$\ew$}
\end{picture}
\end{minipage}
\noindent Then \VR(3.2,0) the bottom right square witnesses $(\uu, \vv) \revR (\ew, \ew)$, which, by Lemma~\ref{L:RevEq}, implies $\uu \eqR \vv$.
\end{proof}

Putting things together, we deduce the practical cancellativity criterion that is the main result of this paper, as stated in the preamble of the paper: 

\begin{prop}\label{P:Crit}
Assume that a monoid presentation~$(\SSS, \RRR)$ satisfies~\eqref{E:Wit1} and \eqref{E:CompatS} for every~$\ss$ in~$\SSS$ and every relation~$\ww \,{=}\, \ww'$ in~$\RRR$. Then a sufficient condition for the monoid~$\MON\SSS\RRR$ to be left cancellative is that there is no relation~$\ss\ww = \ss\ww'$ in~$\RRR$ with $\ww, \ww'$ distinct. 
\end{prop}

 We recall once again that Condition~\eqref{E:Wit1} automatically holds when the considered presentation~$(\SSS, \RRR)$ is homogeneous.

Of course, a symmetric criterion exists for right cancellativity: right noetherianity is to be replaced with left noetherianity, meaning the non-existence of an infinite descending sequence with respect to proper left divisibility, and right reversing grids are to be replaced with their left counterparts, in which one starts from the bottom and right edges and uses the relations to build a rectangular diagram in which the output corresponds to the left and top edges. Note that a right reversing grid is not a left reversing grid, in particular because ``cancellation squares'' are not the same: \VR(9,7) \begin{picture}(16,0)(-3,4)
\pcline{->}(1,10)(9,10)\taput{$\ss$}
\pcline{->}(0,9)(0,1)\tlput{$\ss$}
\pcline{->}(1,0)(9,0)\tbput{$\ew$}
\pcline{->}(10,9)(10,1)\trput{$\ew$}
\end{picture} in a right reversing grid, to be compared with \VR(9,7) \begin{picture}(16,0)(-3,4)
\pcline{->}(1,10)(9,10)\taput{$\ew$}
\pcline{->}(0,9)(0,1)\tlput{$\ew$}
\pcline{->}(1,0)(9,0)\tbput{$\ss$}
\pcline{->}(10,9)(10,1)\trput{$\ss$}
\end{picture} in a left reversing grid.

\begin{rema}\label{R:Embed}
Contrary to Adjan's cancellativity criterion of~\cite{Adj, Rem}, the criterion of Proposition~\ref{P:Crit} does not guarantee that the considered monoid~$\MON\SSS\RRR$ embeds in its universal group, that is, in the group defined, as a group, by the presentation~$(\SSS, \RRR)$, sometimes also called the enveloping group of~$\MON\SSS\RRR$. For instance, consider the monoid~$\MM$ with presentation
\begin{equation}\label{E:Malcev}
\MON{\tta, \ttb, \ttc, \ttd, \tta', \ttb', \ttc', \ttd'}{\tta\ttc = \ttb\ttd, \tta\ttc' = \ttb\ttd', \tta'\ttc = \ttb'\ttd}.
\end{equation}
The monoid~$\MM$ fails to satisfy the first Malcev condition~\cite[Chapter~12, page~310]{ClP} and, therefore, it does not embed in its universal group. However, the presentation of~\eqref{E:Malcev} is eligible for the cancellativity criterion of Proposition~\ref{P:Crit}---but not for Adjan's criterion, since $(\tta, \ttb)$ is a cycle in the left graph. 
\end{rema}

We conclude with one more application of completeness of right reversing, now in terms of common (right) multiples.

\begin{prop}\label{P:CommonMult}
Assume that a monoid presentation~$(\SSS, \RRR)$ satisfies~\eqref{E:Wit1} and \eqref{E:CompatS} for every~$\ss$ in~$\SSS$ and every relation~$\ww \,{=}\, \ww'$ in~$\RRR$. Then two elements~$\aa, \bb$ of~$\MON\SSS\RRR$ respectively represented by words~$\uu$ and~$\vv$ in~$\SSSs$ admit a common right multiple if, and only if, there exists at least one $(\SSS, \RRR)$-grid from~$(\uu, \vv)$; in this case, every common right multiple of~$\aa$ and~$\bb$ is a right multiple of an element represented by~$\uu\vv_1$ and by~$\vv\uu_1$ with~$\uu_1, \vv_1$ satisfying $(\uu, \vv) \rev_\RRR (\uu_1, \vv_1)$.
\end{prop}

\begin{proof}
Assume that there exists a grid from~$(\uu, \vv)$, say $(\uu, \vv) \rev_\RRR (\uu_1, \vv_1)$. By Lemma~\ref{L:RevEq}, this implies $\uu\vv_1 \eqR \vv\uu_1$, which shows that the element of~$\MON\SSS\RRR$ represented by~$\uu\vv_1$ and~$\vv\uu_1$ is a common right multiple of~$\aa$ and~$\bb$. 

\noindent\begin{minipage}{\textwidth}
\rightskip30mm Conversely, \VR(3.5,0) assume that~$\cc$ is a common right multiple of~$\aa$ and~$\bb$: this means that there exist words~$\uu', \vv'$ such that $\cc$ is represented by~$\uu\vv'$ and~$\vv\uu'$, which therefore satisfy $\uu\vv' \eqR \vv\uu'$. Under the assumptions, right reversing is complete for~$(\SSS, \RRR)$, so $(\uu\vv', \vv\uu') \rev_\RRR (\ew, \ew)$ holds. Splitting a reversing grid in four pieces as shown on the right, we see that there exists a grid
\begin{picture}(0,0)(-10,-2)
\pcline{->}(1,20)(9,20)\taput{$\vv$}
\pcline{->}(11,20)(19,20)\taput{$\uu'$}
\pcline{->}(1,10)(9,10)\tbput{$\vv_1$}
\pcline{->}(11,10)(19,10)\tbput{$\vv_2$}
\pcline{->}(1,0)(9,0)\tbput{$\ew$}
\pcline{->}(11,0)(19,0)\tbput{$\ew$}
\pcline{->}(0,19)(0,11)\tlput{$\uu$}
\pcline{->}(0,9)(0,1)\tlput{$\vv'$}
\pcline{->}(10,19)(10,11)\trput{$\uu_1$}
\pcline{->}(10,9)(10,1)\trput{$\uu_2$}
\pcline{->}(20,19)(20,11)\trput{$\ew$}
\pcline{->}(20,9)(20,1)\trput{$\ew$}
\end{picture}
\end{minipage}
\noindent from~$(\uu, \vv)$, \VR(3.5,0) and that the equivalences $\uu' \eqR \uu_1\vv_2$, $\vv_2 \eqR \uu_2$, and $\vv' \eqR \vv_1\uu_2$ are satisfied. The latter show that $\cc$ is a right multiple of the element represented by~$\uu\vv_1$ and~$\vv\uu_1$.
\end{proof}

\begin{coro}\label{C:CommonMult}
Assume that a monoid presentation~$(\SSS, \RRR)$ satisfies the assumptions of Proposition~\ref{P:CommonMult} and, moreover, it is right complemented, \ie, if, for all~$\ss, \ttt$ in~$\SSS$, there is at most one relation $\ss... \,{=}\, \ttt...$ in~$\RRR$. Then two elements~$\aa, \bb$ of~$\MON\SSS\RRR$ respectively represented by~$\uu$ and~$\vv$ in~$\SSSs$ admit a common right multiple if, and only if, $(\uu, \vv) \rev_\RRR (\uu_1, \vv_1)$ holds for some~$\uu_1, \vv_1$; in this case, the element represented by~$\uu\vv_1$ and~$\vv\uu_1$ is a right lcm of~$\aa$ and~$\bb$.
\end{coro}

\begin{proof}
The assumption that $(\SSS, \RRR)$ is right complemented implies that an $(\SSS, \RRR)$-grid from~$(\uu, \vv)$ is unique when it exists. Thus, Proposition~\ref{P:CommonMult} says that every common right multiple of~$\aa$ and~$\bb$ is a right multiple of the element represented by~$\uu\vv_1$. So the latter element, when it exists, is a right lcm of~$\aa$ and~$\bb$.
\end{proof}

Specializing even more, we finally obtain:

\begin{coro}\label{C:CommonMult2}
Assume that a monoid presentation~$(\SSS, \RRR)$ satisfies the assumptions of Proposition~\ref{P:CommonMult} and, moreover, for all~$\ss, \ttt$ in~$\SSS$, there exist~$\ss', \ttt'$ in~$\SSS$ such that $\ss\ttt' \,{=}\, \ttt\ss'$ is a relation of~$\RRR$. Then any two elements of the monoid~$\MON\SSS\RRR$ admit a right lcm.
\end{coro}

\begin{proof}
The presentation is eligible for Corollary~\ref{C:CommonMult}, so we know that any two elements with a common right multiple admit a right lcm. The additional assumption about~$(\SSS, \RRR)$ guarantees that, for all words~$\uu, \vv$ in~$\SSSs$, there exists one $(\SSS, \RRR)$-grid from~$(\uu, \vv)$: indeed, obstructions arise when a relation $\ss... \,{=}\, \ttt...$ is missing, and when the process never terminates because smaller and smaller arrows appear without end. The assumption that there always exist a relation $\ss... \,{=}\, \ttt...$ discards the first obstruction; the assumption that the relations involve words of length~${\le}\,2$ discards the second one. Thus, any two elements of the monoid~$\MON\SSS\RRR$ admit a common right multiple, hence a right lcm.
\end{proof}

\begin{rema}\label{R:Defect}
The cancellativity criterion of Proposition~\ref{P:Crit} subsumes the one established in~\cite{Dgc} in the case of a right complemented presentation. In such a case, there exists at most one $(\SSS, \RRR)$-grid admitting a given source~$(\uu, \vv)$, and, therefore, the output words can be seen as functions of~$\uu$ and~$\vv$. Then, the cancellativity criterion can be stated as a compatibility of the functions in question, called ``complement'', with the equivalence relation~$\eqR$. In our general case, the scheme of the proof remains the same, but one needs to find a different formalism, which makes the extension nontrivial: indeed, whenever the considered presentation contains at least two relations with the same initial letters, there may exist more than one grid with a given source, and complement functions just make no sense. In~\cite{Dgc}, in addition to qualitative aspects, some quantitative results are established, and they can be extended to our current framework. Say that a monoid presentation~$(\SSS, \RRR)$ has \emph{defect}~$\dd$ if, for every~$\ss$ in~$\SSS$, every relation~$\ww \,{=}\, \ww'$ in~$\RRR$, and every $(\SSS, \RRR)$-grid~$\Gamma$ from~$(\ss, \ww)$, there exists an equivalent $(\SSS, \RRR)$-grid~$\Gamma'$ from~$(\ss, \ww')$ such that the sum of the distances between the output words of~$\Gamma$ and~$\Gamma'$ is bounded above by~$\dd$, and $\dd$ is minimal with that property. Then the inductive proof of Proposition~\ref{L:Main} can be adapted to show that, if $(\SSS, \RRR)$ has finite defect~$\dd$ and $\Gamma$ is a grid from~$(\uu, \vv)$, then, for all $\uu' \eqR \uu$ and~$\vv' \eqR \vv$, there exists an equivalent grid~$\Gamma'$ from~$(\uu', \vv')$ such that the distance between the outputs of~$\Gamma$ and~$\Gamma'$ is bounded by an explicit function of the distance between their inputs, actually a double exponential of base~$\dd$. The reader is referred to~\cite{Dgc} to fill in the details.
\end{rema}

\section{Applications to variants of braid monoids}\label{S:Braids}

As an application of the results of Section~\ref{S:Criterion}, we now establish that the monoids of colored braids, which are extensions of the classical Artin braid monoids, admit cancellation.

\subsection{Braids with colored crossings}\label{SS:Colored}

We mentioned in Example~\ref{X:Braid} that, for $\nn \ge 1$, the standard $\nn$-strand monoid~$\BP\nn$ is the monoid presented by~\eqref{E:BraidPres}. We recall, for instance from~\cite{Bir}, that, under interpreting~$\sig\ii$ as the elementary crossing that exchanges the strands at positions~$\ii$ and~$\ii + 1$ as in
$$
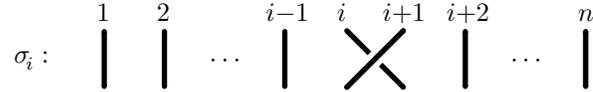
\begin{picture}(64,12)(0,0)
\put(-12,3.5){$\sig\ii : $}
\psline[linewidth=2pt]{cc-cc}(0,8)(0,0)\put(-1,9){$1$}
\psline[linewidth=2pt]{cc-cc}(8,8)(8,0)\put(7,9){$2$}
\put(14,3){$\cdots$}
\psline[linewidth=2pt]{cc-cc}(24,8)(24,0)\put(21.5,9){$\ii{-}1$}
\psline[linewidth=2pt]{cc-cc}(32,8)(40,0)\put(31,9){$\ii$}
\psline[linewidth=2pt,border=2pt]{cc-cc}(40,8)(32,0)\put(37,9){$\ii{+}1$}
\psline[linewidth=2pt]{cc-cc}(48,8)(48,0)\put(45.5,9){$\ii{+}2$}
\put(54,3){$\cdots$}
\psline[linewidth=2pt]{cc-cc}(64,8)(64,0)\put(63,9){$\nn$}
\end{picture}$$
the monoid~$\BP\nn$ is the monoid of isotopy classes of positive $\nn$-strand braid diagrams. We now consider an extension of the monoid~$\BP\nn$:

\begin{defi}
For $\nn \ge 1$ and $\CC$ a nonempty set, the monoid of \emph{positive $\CC$-colored braids} is the monoid with presentation
\begin{equation}\label{E:ColBraidPres}
\BP{\nn, \CC}:= \bigg\langle \sigg\ii\aa ( \ii \le \nn, \aa \in \CC ) \ \bigg\vert\ 
\begin{matrix}
\sigg\ii\aa \sigg\jj\bb = \sigg\jj\bb \sigg\ii\aa 
&\text{for} &\vert i-j \vert\ge 2\\
\sigg\ii\aa \sigg\jj\bb \sigg\ii\cc = \sigg\jj\cc \sigg\ii\bb \sigg\jj\aa 
&\text{for} &\vert i-j \vert = 1
\end{matrix}
\ \bigg\rangle^{\!+}.
\end{equation}
\end{defi}

The idea is that the generator~$\sigg\ii\aa$ corresponds (as usual) to a crossing at positions~$\ii$ and~$\ii + 1$ with, in addition, an attached ``color''~$\aa$ in~$\CC$. The relations of~\eqref{E:ColBraidPres} are then natural if we imagine that the colors are connected with the names, or initial positions, of the strands (as opposed to the current positions). Typically, we may think of taking for~$\CC$ the set of all (unordered) pairs in~$\{1 \wdots \nn\}$, the meaning of the crossing~$\sigg\ii{\pp, \qq}$ being ``the strands starting at positions~$\pp$ and~$\qq$ cross at position~$\ii$'', see Figure~\ref{F:ColorCross}.

\begin{figure}[htb]
\begin{picture}(16,25)(0,2)
\psline[linewidth=2pt]{cc-cc}(0,24)(16,8)(16,0)\put(-1,25.5){$\pp$}
\psline[linewidth=2pt,border=2pt]{cc-cc}(8,24)(0,16)(0,8)(8,0)\put(7,25.5){$\qq$}
\psline[linewidth=2pt,border=2pt]{cc-cc}(16,24)(16,16)(0,0)\put(15,25.5){$\rr$}
\put(-7,19){$\sigg\ii{\pp,\qq}$}
\put(16,11){$\sigg{\ii+1}{\pp,\rr}$}
\put(-7,3){$\sigg\ii{\qq,\rr}$}
\put(-7,3){$\sigg\ii{\qq,\rr}$}
\put(28,11){$\sim$}
\end{picture}
\hspace{25mm}
\begin{picture}(21,25)(0,2)
\psline[linewidth=2pt]{cc-cc}(0,24)(0,16)(16,0)\put(-1,25.5){$\pp$}
\psline[linewidth=2pt,border=2pt]{cc-cc}(8,24)(16,16)(16,8)(8,0)\put(7,25.5){$\qq$}
\psline[linewidth=2pt,border=2pt]{cc-cc}(16,24)(0,8)(0,0)\put(15,25.5){$\rr$}
\put(16,19){$\sigg{\ii+1}{\qq,\rr}$}
\put(-7,11){$\sigg\ii{\pp,\rr}$}
\put(16,3){$\sigg{\ii+1}{\pp,\qq}$}
\end{picture}
\caption{\sf Colored braid relation: if we give~$\sigg\ii{\pp, \qq}$ the meaning ``the strands starting at positions~$\pp$ and~$\qq$ cross at position~$\ii$'', that is, if we take into account the names (origins) of the strand that cross, then the relations of~\eqref{E:ColBraidPres} appear naturally.}
\label{F:ColorCross}
\end{figure}
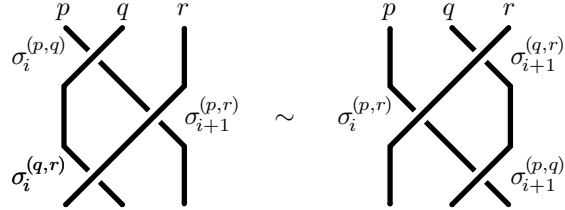

Of course, when the colour set~$\CC$ is a singleton, we can forget about colours, and the monoid~$\BP{\nn, \CC}$ is simply the $\nn$-strand monoid---which is known to be cancellative since Garside~\cite{Gar}. By contrast, for $\card\CC \ge 2$ and $\nn \ge 3$, the presentation of~\eqref{E:ColBraidPres} is not complemented (for some generators~$\ss, \ttt$, there is more than one relation of the type $\ss... = \ttt...$ in the presentation), and no simple criterion seems to apply. Here we shall prove:

\begin{prop}\label{P:CBCancel}
The monoid~$\BP{\nn, \CC}$ admits left and right cancellation.
\end{prop}

The proof consists of applying the criterion of Proposition~\ref{P:Crit}, namely considering all generators~$\sigg\ii\aa$ and all relations~$\ww \,{=}\, \ww'$ of~\eqref{E:ColBraidPres}, and checking that, for every reversing grid built from~$\sigg\ii\aa$ and~$\ww$, there exists an equivalent reversing grid built from~$\sigg\ii\aa$ and~$\ww'$, and \textit{vice versa}. We shall see that there are only two critical cases, with all other cases either reducing to them or being trivial.

\begin{lemm}\label{L:Case1}
Property~\eqref{E:CompatS} holds for $\sigg1\aa$ and the relation $\sigg2\bb \sigg3\cc \sigg2\dd = \sigg3\dd \sigg2\cc \sigg3\bb$.
\end{lemm}

\begin{proof}
We look at all possible grids from~$(\sigg1\aa, \sigg2\bb \sigg3\cc \sigg2\dd)$, and exhibit for each of them an equivalent grid from~$(\sigg1\aa, \sigg3\dd \sigg2\cc \sigg3\bb)$, and conversely. First, exhaustively inspecting the presentation shows that the valid grids from~$(\sigg1\aa, \sigg2\bb \sigg3\cc \sigg2\dd)$ are the following grids, where $\ee$ and~$\ff$ are arbitrary elements of the color set~$\CC$:
\begin{equation}\label{E:Case11}
\VR(21.5,19.5)\begin{picture}(60,0)(0,13)
\psset{yunit=1.2mm}
\pcline{->}(1,24)(19,24)\taput{$\sigg2\bb$}
\pcline{->}(21,24)(39,24)\taput{$\sigg3\cc$}
\pcline{->}(41,24)(59,24)\taput{$\sigg2\dd$}
\pcline{->}(21,12)(39,12)\taput{$\sigg3\cc$}
\pcline{->}(41,12)(49,12)\taput{$\sigg2\ff$}
\pcline{->}(51,12)(59,12)\taput{$\sigg1\dd$}
\pcline{->}(41,6)(49,6)\taput{$\ew$}
\pcline{->}(51,6)(59,6)\taput{$\sigg1\dd$}
\pcline{->}(1,0)(9,0)\tbput{$\sigg2\ee$}
\pcline{->}(11,0)(19,0)\tbput{$\sigg1\bb$}
\pcline{->}(21,0)(29,0)\tbput{$\sigg3\ff$}
\pcline{->}(31,0)(39,0)\tbput{$\sigg2\cc$}
\pcline{->}(41,0)(49,0)\tbput{$\ew$}
\pcline{->}(51,0)(59,0)\tbput{$\sigg1\dd$}
\pcline{->}(0,23)(0,1)\tlput{$\sigg1\aa$}
\pcline{->}(20,23)(20,13)\tlput{$\sigg1\ee$}
\pcline{->}(20,11)(20,1)\tlput{$\sigg2\aa$}
\pcline{->}(40,23)(40,13)\tlput{$\sigg1\ee$}
\pcline{->}(40,11)(40,7)\tlput{$\sigg2\ff$}
\pcline{->}(40,5)(40,1)\tlput{$\sigg3\aa$}
\pcline{->}(50,11)(50,7)\tlput{$\ew$}
\pcline{->}(50,5)(50,1)\tlput{$\sigg3\aa$}
\pcline{->}(60,23)(60,19)\trput{$\sigg1\ff$}
\pcline{->}(60,17)(60,13)\trput{$\sigg2\ee$}
\pcline{->}(60,11)(60,7)\trput{$\ew$}
\pcline{->}(60,5)(60,1)\trput{$\sigg3\aa$}
\put(29,7){$*$}
\put(49,21){$*$}
\end{picture}
\end{equation}
A priori, one might use different colors~$\ff, \ff'$ in the squares marked~$*$, but~$\ff \not= \ff'$ leads to a grid that cannot be completed, since there is no relation~$\sigg2\ff... \,{=}\, \sigg2{\ff'}...$ in~\eqref{E:ColBraidPres}. Now, consider the following valid grid:
\begin{equation}\label{E:Case12}
\VR(16.5,14.5)\begin{picture}(60,0)(0,10)
\psset{yunit=1.2mm}
\pcline{->}(1,18)(19,18)\taput{$\sigg3\dd$}
\pcline{->}(21,18)(39,18)\taput{$\sigg2\cc$}
\pcline{->}(41,18)(59,18)\taput{$\sigg3\bb$}
\pcline{->}(41,12)(59,12)\taput{$\sigg3\bb$}
\pcline{->}(1,0)(19,0)\tbput{$\sigg3\dd$}
\pcline{->}(21,0)(29,0)\tbput{$\sigg2\ff$}
\pcline{->}(31,0)(39,0)\tbput{$\sigg1\cc$}
\pcline{->}(41,0)(49,0)\tbput{$\sigg3\ee$}
\pcline{->}(51,0)(59,0)\tbput{$\sigg2\bb$}
\pcline{->}(0,17)(0,1)\tlput{$\sigg1\aa$}
\pcline{->}(20,17)(20,1)\tlput{$\sigg1\aa$}
\pcline{->}(40,17)(40,13)\tlput{$\sigg1\ff$}
\pcline{->}(40,11)(40,1)\tlput{$\sigg2\aa$}
\pcline{->}(60,17)(60,13)\trput{$\sigg1\ff$}
\pcline{->}(60,11)(60,7)\trput{$\sigg2\ee$}
\pcline{->}(60,5)(60,1)\trput{$\sigg3\aa$}
\end{picture}
\end{equation}
The right edges of~\eqref{E:Case11} and~\eqref{E:Case12} both yield the word~$\sigg1\ff \sigg2\ee \sigg3\aa$. For the bottom edges, using~$\eqR$ for the congruence generated by the relations of~\eqref{E:ColBraidPres}, we find
\begin{align*}
\sigg3\dd \sigg2\ff \sigg1\cc \sigg3\ee \sigg2\bb 
\eqR \sigg3\dd \sigg2\ff &\sigg3\ee \sigg1\cc \sigg2\bb
\eqR \sigg2\ee \sigg3\ff \sigg2\dd \sigg1\cc \sigg2\bb\\
&\eqR \sigg2\ee \sigg3\ff \sigg1\bb \sigg2\cc \sigg1\dd
\eqR \sigg2\ee \sigg1\bb \sigg3\ff \sigg2\cc \sigg1\dd,
\end{align*}
which shows that \eqref{E:Case11} and~\eqref{E:Case12} are equivalent grids.

Conversely, starting from~$\sigg1\aa$ and~$\sigg3\dd \sigg2\cc \sigg3\bb$, the only possible reversing grids have the form of~\eqref{E:Case12} for some~$\ee$ and~$\ff$ in~$\CC$, and then~\eqref{E:Case11} provides the expected equivalent grid.
\end{proof}

\begin{lemm}\label{L:Case2}
Property~\eqref{E:CompatS} holds for $\sigg2\aa$ and the relation $\sigg1\bb \sigg3\cc = \sigg3\cc \sigg1\bb$.
\end{lemm}

\begin{proof}
As in Lemma~\ref{L:Case1}, we look at all reversing grids from $(\sigg2\aa, \sigg1\bb \sigg3\cc)$, and exhibit an equivalent grid from $(\sigg2\aa, \sigg3\cc \sigg1\bb)$, and conversely. The grids from $(\sigg2\aa, \sigg1\bb \sigg3\cc)$ are the following grids, with~$\dd$, $\ee$, $\ff$ arbitrary in~$\CC$:
\begin{equation}\label{E:Case21}
\VR(18.5,16.5)\begin{picture}(50,0)(0,11)
\pcline{->}(1,24)(19,24)\taput{$\sigg1\bb$}
\pcline{->}(21,24)(49,24)\taput{$\sigg3\cc$}
\pcline{->}(21,12)(29,12)\taput{$\sigg3\ee$}
\pcline{->}(31,12)(49,12)\taput{$\sigg2\cc$}
\pcline{->}(1,0)(9,0)\tbput{$\sigg1\dd$}
\pcline{->}(11,0)(19,0)\tbput{$\sigg2\bb$}
\pcline{->}(21,0)(29,0)\tbput{$\sigg3\ee$}
\pcline{->}(31,0)(39,0)\tbput{$\sigg2\ff$}
\pcline{->}(41,0)(49,0)\tbput{$\sigg1\cc$}
\pcline{->}(0,23)(0,1)\tlput{$\sigg2\aa$}
\pcline{->}(20,23)(20,13)\tlput{$\sigg2\dd$}
\pcline{->}(20,11)(20,1)\tlput{$\sigg1\aa$}
\pcline{->}(30,11)(30,1)\trput{$\sigg1\aa$}
\pcline{->}(50,23)(50,19)\trput{$\sigg2\ee$}
\pcline{->}(50,17)(50,13)\trput{$\sigg3\dd$}
\pcline{->}(50,11)(50,7)\trput{$\sigg1\ff$}
\pcline{->}(50,5)(50,1)\trput{$\sigg2\aa$}
\end{picture}
\end{equation}
On the other hand, the following grid is also valid:
\begin{equation}\label{E:Case22}
\VR(18.5,16.5)\begin{picture}(50,0)(0,11)
\pcline{->}(1,24)(19,24)\taput{$\sigg3\cc$}
\pcline{->}(21,24)(49,24)\taput{$\sigg1\bb$}
\pcline{->}(21,12)(29,12)\taput{$\sigg1\ee$}
\pcline{->}(31,12)(49,12)\taput{$\sigg2\bb$}
\pcline{->}(1,0)(9,0)\tbput{$\sigg3\ff$}
\pcline{->}(11,0)(19,0)\tbput{$\sigg2\cc$}
\pcline{->}(21,0)(29,0)\tbput{$\sigg1\ee$}
\pcline{->}(31,0)(39,0)\tbput{$\sigg2\dd$}
\pcline{->}(41,0)(49,0)\tbput{$\sigg3\bb$}
\pcline{->}(0,23)(0,1)\tlput{$\sigg2\aa$}
\pcline{->}(20,23)(20,13)\tlput{$\sigg2\ff$}
\pcline{->}(20,11)(20,1)\tlput{$\sigg3\aa$}
\pcline{->}(30,11)(30,1)\trput{$\sigg3\aa$}
\pcline{->}(50,23)(50,19)\trput{$\sigg2\ee$}
\pcline{->}(50,17)(50,13)\trput{$\sigg1\ff$}
\pcline{->}(50,11)(50,7)\trput{$\sigg3\dd$}
\pcline{->}(50,5)(50,1)\trput{$\sigg2\aa$}
\end{picture}
\end{equation}
The right edges of~\eqref{E:Case21} and~\eqref{E:Case22} correspond to equivalent words, since we have 
$$\sigg2\ee \sigg1\ff \sigg3\dd \sigg2\aa \eqR \sigg2\ee \sigg3\dd \sigg1\ff \sigg2\aa.$$
Similarly, we find for the bottom edges
$$\sigg3\ff \sigg2\cc \sigg1\ee \sigg2\dd \sigg3\bb 
\eqR \sigg3\ff \sigg1\dd \sigg2\ee \sigg1\cc \sigg3\bb
\eqR \sigg1\dd \sigg3\ff \sigg2\ee \sigg3\bb \sigg1\cc 
\eqR \sigg1\dd \sigg2\bb \sigg3\ee \sigg2\ff \sigg1\cc.$$
Hence \eqref{E:Case21} and~\eqref{E:Case22} are equivalent grids.

Conversely, the possible reversing grids from~$(\sigg2\aa, \sigg3\cc \sigg1\bb)$ have the form of~\eqref{E:Case22} for some~$\ff$, $\ee$, and~$\dd$ in~$\CC$, and then~\eqref{E:Case21} provides the expected equivalent grid.
\end{proof}

\begin{lemm}\label{L:AllCases}
Property~\eqref{E:CompatS} holds for all~$\sigg\ii\aa$ and all relations of~\eqref{E:ColBraidPres}.
\end{lemm}

\begin{proof}
First consider the case of~$\sigg\ii\aa$ and a relation $\sigg\jj\bb \sigg\kk\cc \sigg\jj\dd \,{=}\, \sigg\kk\dd \sigg\jj\cc \sigg\kk\bb$ with, say, $\kk = \jj + 1$. For $\ii \le \jj - 2$, every grid from~$(\sigg\ii\aa, \sigg\jj\bb \sigg\kk\cc \sigg\jj\dd)$ is a commutation grid, namely a grid in which every tile corresponds to a commutation relation $\ss\ttt \,{=}\, \ttt\ss$, and then there exists an equivalent grid from~$(\sigg\ii\aa, \sigg\kk\dd \sigg\jj\cc \sigg\kk\bb)$ as shown below:
\begin{equation}\label{E:Case3}
\VR(9.5,7.5)\begin{picture}(30,0)(0,3)
\pcline{->}(1,8)(9,8)\taput{$\sigg\jj\bb$}
\pcline{->}(11,8)(19,8)\taput{$\sigg\kk\cc$}
\pcline{->}(21,8)(29,8)\taput{$\sigg\jj\dd$}
\pcline{->}(1,0)(9,0)\tbput{$\sigg\jj\bb$}
\pcline{->}(11,0)(19,0)\tbput{$\sigg\kk\cc$}
\pcline{->}(21,0)(29,0)\tbput{$\sigg\jj\dd$}
\pcline{->}(0,7)(0,1)\tlput{$\sigg\ii\aa$}
\pcline{->}(10,7)(10,1)\trput{$\sigg\ii\aa$}
\pcline{->}(20,7)(20,1)\trput{$\sigg\ii\aa$}
\pcline{->}(30,7)(30,1)\trput{$\sigg\ii\aa$}
\end{picture}
\hspace{20mm}
\begin{picture}(30,0)(0,3)
\pcline{->}(1,8)(9,8)\taput{$\sigg\kk\dd$}
\pcline{->}(11,8)(19,8)\taput{$\sigg\jj\cc$}
\pcline{->}(21,8)(29,8)\taput{$\sigg\kk\bb$}
\pcline{->}(1,0)(9,0)\tbput{$\sigg\kk\dd$}
\pcline{->}(11,0)(19,0)\tbput{$\sigg\jj\cc$}
\pcline{->}(21,0)(29,0)\tbput{$\sigg\kk\bb$}
\pcline{->}(0,7)(0,1)\tlput{$\sigg\ii\aa$}
\pcline{->}(10,7)(10,1)\trput{$\sigg\ii\aa$}
\pcline{->}(20,7)(20,1)\trput{$\sigg\ii\aa$}
\pcline{->}(30,7)(30,1)\trput{$\sigg\ii\aa$}
\end{picture}
\end{equation}
The case when we start with a grid from~$(\sigg\ii\aa, \sigg\kk\dd \sigg\jj\cc \sigg\kk\bb)$ is similar.

The case $\ii = \jj - 1$ corresponds to Lemma~\ref{L:Case1} for~$\ii = 1$, and the general case is similar, since the relations of~\eqref{E:ColBraidPres} are invariant under shifting the indices.

Next, assume $\ii = \jj$. Then a grid from~$(\sigg\ii\aa, \sigg\jj\bb \sigg\kk\cc \sigg\jj\dd)$ exists only for $\aa = \bb$, and then it is as in the left hand diagram below, in which case the right hand diagram provides an equivalent grid for $\ee = \cc$:
\begin{equation}\label{E:Case4}
\VR(13,11)\begin{picture}(30,0)(0,3)
\pcline{->}(1,8)(9,8)\taput{$\sigg\jj\bb$}
\pcline{->}(11,8)(19,8)\taput{$\sigg\kk\cc$}
\pcline{->}(21,8)(29,8)\taput{$\sigg\jj\dd$}
\pcline{->}(1,0)(9,0)\tbput{$\ew$}
\pcline{->}(11,0)(19,0)\tbput{$\sigg\kk\cc$}
\pcline{->}(21,0)(29,0)\tbput{$\sigg\jj\dd$}
\pcline{->}(0,7)(0,1)\tlput{$\sigg\ii\aa$}
\pcline{->}(10,7)(10,1)\trput{$\ew$}
\pcline{->}(20,7)(20,1)\trput{$\ew$}
\pcline{->}(30,7)(30,1)\trput{$\ew$}
\end{picture}
\hspace{20mm}
\begin{picture}(40,0)(0,6)
\pcline{->}(1,16)(19,16)\taput{$\sigg\kk\dd$}
\pcline{->}(21,16)(29,16)\taput{$\sigg\jj\cc$}
\pcline{->}(31,16)(39,16)\taput{$\sigg\kk\bb$}
\pcline{->}(21,8)(29,8)\taput{$\ew$}
\pcline{->}(31,8)(39,8)\taput{$\sigg\kk\bb$}
\pcline{->}(1,0)(9,0)\tbput{$\sigg\kk\ee$}
\pcline{->}(11,0)(19,0)\tbput{$\sigg\jj\dd$}
\pcline{->}(21,0)(29,0)\tbput{$\ew$}
\pcline{->}(31,0)(39,0)\tbput{$\ew$}
\pcline{->}(0,15)(0,1)\tlput{$\sigg\ii\aa$}
\pcline{->}(20,15)(20,9)\tlput{$\sigg\jj\ee$}
\pcline{->}(20,7)(20,1)\tlput{$\sigg\kk\aa$}
\pcline{->}(30,15)(30,9)\tlput{$\ew$}
\pcline{->}(30,7)(30,1)\tlput{$\sigg\kk\aa$}
\pcline{->}(40,15)(40,9)\trput{$\ew$}
\pcline{->}(40,7)(40,1)\trput{$\ew$}
\end{picture}
\end{equation}
In the other direction, the only possible grids from~$(\sigg\ii\aa, \sigg\kk\dd \sigg\jj\cc \sigg\kk\bb)$ correspond to the right hand diagram in~\eqref{E:Case4} with $\ee = \cc$ and~$\aa = \bb$, in which case the left diagram provides the required equivalent grid.

The case $\ii = \jj + 1 = \kk$ is symmetric to $\ii = \jj$. Similarly, the case $\ii = \jj + 2 = \kk + 1$ is symmetric to $\ii = \jj - 1$. Finally, the cases $\ii \ge \kk + 2$ are symmetric to $\ii \le \jj - 2$. 

We now consider the case of~$\sigg\ii\aa$ and a relation $\sigg\jj\bb \sigg\kk\cc \,{=}\, \sigg\kk\cc \sigg\jj\bb$ with, say, $\kk \ge\nobreak \jj +\nobreak 2$. The case $\ii \le \jj - 2$ is similar to that of~\eqref{E:Case3}, with commutation grids.

Assume $\ii = \jj - 1$. Then the grids from~$(\sigg\ii\aa, \sigg\jj\bb \sigg\kk\cc)$ are as on the left diagram below, with~$\dd$ arbitrary in~$\CC$, and the right diagram then provides the expected equivalent grid from~$(\sigg\ii\aa, \sigg\kk\cc \sigg\jj\bb )$:
\begin{equation}\label{E:Case5}
\VR(14,12)\begin{picture}(30,0)(0,7)
\pcline{->}(1,16)(19,16)\taput{$\sigg\jj\bb$}
\pcline{->}(21,16)(29,16)\taput{$\sigg\kk\cc$}
\pcline{->}(21,8)(29,8)\taput{$\sigg\kk\cc$}
\pcline{->}(1,0)(9,0)\tbput{$\sigg\jj\dd$}
\pcline{->}(11,0)(19,0)\tbput{$\sigg\ii\bb$}
\pcline{->}(21,0)(29,0)\tbput{$\sigg\kk\cc$}
\pcline{->}(0,15)(0,1)\tlput{$\sigg\ii\aa$}
\pcline{->}(20,15)(20,9)\tlput{$\sigg\ii\dd$}
\pcline{->}(20,7)(20,1)\tlput{$\sigg\jj\aa$}
\pcline{->}(30,15)(30,9)\trput{$\sigg\ii\dd$}
\pcline{->}(30,7)(30,1)\trput{$\sigg\jj\aa$}
\end{picture}
\hspace{20mm}
\begin{picture}(30,0)(0,6)
\pcline{->}(1,16)(9,16)\taput{$\sigg\kk\cc$}
\pcline{->}(11,16)(29,16)\taput{$\sigg\jj\bb$}
\pcline{->}(1,0)(9,0)\tbput{$\sigg\kk\cc$}
\pcline{->}(11,0)(19,0)\tbput{$\sigg\jj\dd$}
\pcline{->}(21,0)(29,0)\tbput{$\sigg\ii\bb$}
\pcline{->}(0,15)(0,1)\tlput{$\sigg\ii\aa$}
\pcline{->}(10,15)(10,1)\trput{$\sigg\ii\aa$}
\pcline{->}(30,15)(30,9)\trput{$\sigg\ii\dd$}
\pcline{->}(30,7)(30,1)\trput{$\sigg\jj\aa$}
\end{picture}
\end{equation}
In the other direction, the only grids from~$(\sigg\ii\aa, \sigg\kk\cc \sigg\jj\bb )$ are those shown in the right diagram of~\eqref{E:Case5} with~$\dd$ arbitrary in~$\CC$, and the left diagram then provides the expected equivalent grid.

The case $\ii = \jj$ is almost trivial: grids may exist only for $\aa = \bb$, and then they take the form
\begin{equation}\label{E:Case6}
\VR(9,8)\begin{picture}(20,0)(0,3)
\pcline{->}(1,8)(9,8)\taput{$\sigg\jj\bb$}
\pcline{->}(11,8)(19,8)\taput{$\sigg\kk\cc$}
\pcline{->}(1,0)(9,0)\tbput{$\ew$}
\pcline{->}(11,0)(19,0)\tbput{$\sigg\kk\cc$}
\pcline{->}(0,7)(0,1)\tlput{$\sigg\ii\aa$}
\pcline{->}(10,7)(10,1)\trput{$\ew$}
\pcline{->}(20,7)(20,1)\trput{$\ew$}
\end{picture}
\hspace{20mm}
\begin{picture}(20,0)(0,3)
\pcline{->}(1,8)(9,8)\taput{$\sigg\kk\cc$}
\pcline{->}(11,8)(19,8)\taput{$\sigg\jj\bb$}
\pcline{->}(1,0)(9,0)\tbput{$\sigg\kk\cc$}
\pcline{->}(11,0)(19,0)\tbput{$\ew$}
\pcline{->}(0,7)(0,1)\tlput{$\sigg\ii\aa$}
\pcline{->}(10,7)(10,1)\trput{$\sigg\ii\aa$}
\pcline{->}(20,7)(20,1)\trput{$\ew$}
\end{picture}
\end{equation}

Next, assume $\jj + 1 \le \ii \le \kk - 1$. If $\kk = \jj + 2$ holds, typically $\ii = 2$, $\jj = 1$, $\kk = 3$, we are, up to a shifting of the indices, in the situation of Lemma~\ref{L:Case2}, and so~\eqref{E:CompatS} is guaranteed. Otherwise, either $\ii$ is adjacent to exactly one of~$\jj$ or~$\kk$, and the situation is that of~\eqref{E:Case5}, or $\ii$ is at distance at least~$2$ from both~$\jj$ and~$\kk$, and the situation is that of~\eqref{E:Case4}. Finally, the cases of $\ii = \kk$, $\ii = \kk + 1$, and $\ii \ge \kk + 2$ are symmetric to those of~\eqref{E:Case6}, \eqref{E:Case5}, and~\eqref{E:Case4}, respectively. Thus, all cases have been successfully treated.
\end{proof}

We can now easily complete the proof of Proposition~\ref{P:CBCancel}:

\begin{proof}[Proof of Proposition~\ref{P:CBCancel}]
The presentation~\eqref{E:ColBraidPres} is eligible for the criterion of Proposition~\ref{P:Crit}. Indeed, all relations are of the form $\ww \,{=}\, \ww'$ with $\ww, \ww'$ of the same length. Hence the monoid~$\BP{\nn, \CC}$ is right noetherian. By Proposition~\ref{P:Crit} and Lemma~\ref{L:AllCases}, right reversing is complete for~\eqref{E:ColBraidPres}. Hence, as the presentation contains no relation contradicting left cancellation, the monoid admits left cancellation. Finally, the symmetry of the relations guarantees that the identity map on the generators induces an anti-automorphism of the monoid, and, therefore, right cancellativity automatically follows from left cancellativity. 
\end{proof}

Inspecting the proofs above shows that, in the worst cases, the combinatorial distance between the outputs of the old and the new grids is at most~$5$, so, according to the terminology sketched in Remark~\ref{R:Defect}, the defect of the presentation~\eqref{E:ColBraidPres} is~$5$, which could be used to obtain explicit upper bounds on the number or reversing steps needed to possibly establish the equivalence of words.

As mentioned in Remark~\ref{R:Embed}, our current approach says nothing about the embeddability of the involved monoid in a group. So the obvious question after Proposition~\ref{P:CBCancel} is

\begin{ques}\label{Q:Embed}
Does the monoid~$\BP{\nn, \CC}$ embed in its universal group?
\end{ques}

A classical sufficient condition is provided by Ore's theorem~\cite{Ore} stating in the current context that a cancellative monoid~$\MM$ in which any two elements admit a common right multiple embeds in its universal group, which, in addition, is then a group of right fractions for~$\MM$. This applies for instance to the monoid~$\BP\nn$. However, for~$\card\CC \ge 2$, the monoid~$\BP{\nn, \CC}$ admits no common multiple: for $\aa \not= \bb$, the elements~$\sigg1\aa$ and~$\sigg1\bb$ admit no common right (or left) multiple, since there is no valid reversing grid from~$(\sigg1\aa, \sigg1\bb)$. In~\cite{Dit}, the embeddability criterion of Ore's theorem is extended to cancellative monoids with no nontrivial invertible elements that satisfy the following ``$3$-Ore condition'':
\begin{equation}\label{E:3Ore}
\parbox{110mm}{any three elements of~$\MM$ which pairwise admit a common right multiple admit a common right multiple, and similarly for left multiples,}
\end{equation}
provided any two elements of~$\MM$ admit a left and a right gcd, \ie, greatest lower bounds with respect to left and right division. For $\card\CC \ge 2$, the monoid~$\BP{\nn, \CC}$ does not admit gcds: for instance, for~$\aa \not= \bb$, the elements~$\sigg1\aa$ and~$\sigg2\aa$ left divide both~$\sigg1\aa \sigg2\aa \sigg1\aa$ and~$\sigg1\aa \sigg2\bb \sigg1\aa$, but no common multiple of~$\sigg1\aa$ and~$\sigg2\aa$ left divides the above elements. This leads to two new questions:

\begin{ques}
Does the monoid~$\BP{\nn, \CC}$ satisfy the $3$-Ore condition~\eqref{E:3Ore}?
\end{ques} 

\begin{ques}
Is the $3$-Ore condition~\eqref{E:3Ore} sufficient for implying the embeddability of a monoid in its universal group in the case of a cancellative monoid that need not admit gcds?
\end{ques}

\subsection{A variant}

In~\cite{BGL}, the authors consider a variant of the monoid~$\BP{\nn, \CC}$ with the same generators but with a restricted list of relations:

\begin{defi}\label{D:BGL}
For $\nn \ge 1$ and $\CC$ a nonempty set, the monoid of \emph{restricted positive $\CC$-colored braids} is the monoid with presentation
\begin{equation}\label{E:BGL}
\BPr{\nn, \CC}:= \bigg\langle \sigg\ii\aa ( \ii \le \nn, \aa \in \CC ) \bigg\vert\ 
\begin{matrix}
\sigg\ii\aa \sigg\jj\bb = \sigg\jj\bb \sigg\ii\aa 
&\text{for} &\vert i-j \vert\ge 2\\
\sigg\ii\aa \sigg\jj\aa \sigg\ii\bb = \sigg\jj\bb \sigg\ii\aa \sigg\jj\aa 
&\text{for} &\vert i-j \vert = 1
\end{matrix}
\ \bigg\rangle^{\!+}.
\end{equation}
\end{defi}

All relations of~\eqref{E:BGL} are relations of~\eqref{E:ColBraidPres}, but, in the ``Yang-Baxter'' relations, the median color must be equal to one of the extremal colors. Thus the monoid~$\BP{\nn, \CC}$ is a quotient of the monoid~$\BPr{\nn, \CC}$. The authors of~\cite{BGL} ask whether the monoid~$\BPr{\nn, \CC}$ is cancellative. Frustratingly, the criterion of Proposition~\ref{P:Crit} cannot be applied:

\begin{fact}\label{F:BGL}
 If $\CC$ has at least two elements, right reversing is not complete for the presentation~\eqref{E:BGL}.
\end{fact}

\begin{proof}
 Let $\aa, \bb, \cc$ be elements of~$\CC$ satisfying $\aa \not= \bb$ and $\aa \not= \cc$. Then Property~\eqref{E:CompatS} fails for $\sigg1\aa$ and the relation $\sigg2\bb \sigg3\cc \sigg2\cc \,{=}\, \sigg3\cc \sigg2\cc \sigg3\bb$. Indeed, we have the following valid grid
\begin{equation}\label{E:Case12BGL}
\VR(18,15)\begin{picture}(60,0)(0,10)
\psset{yunit=1.2mm}
\pcline{->}(1,18)(19,18)\taput{$\sigg3\cc$}
\pcline{->}(21,18)(39,18)\taput{$\sigg2\cc$}
\pcline{->}(41,18)(59,18)\taput{$\sigg3\bb$}
\pcline{->}(41,12)(59,12)\taput{$\sigg3\bb$}
\pcline{->}(1,0)(19,0)\tbput{$\sigg3\cc$}
\pcline{->}(21,0)(29,0)\tbput{$\sigg2\aa$}
\pcline{->}(31,0)(39,0)\tbput{$\sigg1\cc$}
\pcline{->}(41,0)(49,0)\tbput{$\sigg3\bb$}
\pcline{->}(51,0)(59,0)\tbput{$\sigg2\bb$}
\pcline{->}(0,17)(0,1)\tlput{$\sigg1\aa$}
\pcline{->}(20,17)(20,1)\tlput{$\sigg1\aa$}
\pcline{->}(40,17)(40,13)\tlput{$\sigg1\aa$}
\pcline{->}(40,11)(40,1)\tlput{$\sigg2\aa$}
\pcline{->}(60,17)(60,13)\trput{$\sigg1\aa$}
\pcline{->}(60,11)(60,7)\trput{$\sigg2\bb$}
\pcline{->}(60,5)(60,1)\trput{$\sigg3\aa$}
\end{picture}
\end{equation}
and there may exist no equivalent grid from~$(\sigg1\aa, \sigg2\bb\sigg3\cc \sigg2\cc)$. Indeed, according to what was seen in the proof of Lemma~\ref{L:Case1}, the only possible form for such a grid would be
\begin{equation}\label{E:Case11BGL}
\VR(21,18)\begin{picture}(60,0)(0,13)
\psset{yunit=1.2mm}
\pcline{->}(1,24)(19,24)\taput{$\sigg2\bb$}
\pcline{->}(21,24)(39,24)\taput{$\sigg3\cc$}
\pcline{->}(41,24)(59,24)\taput{$\sigg2\cc$}
\pcline{->}(21,12)(39,12)\taput{$\sigg3\cc$}
\pcline{->}(41,12)(49,12)\taput{$\sigg2\yy$}
\pcline{->}(51,12)(59,12)\taput{$\sigg1\cc$}
\pcline{->}(41,6)(49,6)\taput{$\ew$}
\pcline{->}(51,6)(59,6)\taput{$\sigg1\cc$}
\pcline{->}(1,0)(9,0)\tbput{$\sigg2\xx$}
\pcline{->}(11,0)(19,0)\tbput{$\sigg1\bb$}
\pcline{->}(21,0)(29,0)\tbput{$\sigg3\yy$}
\pcline{->}(31,0)(39,0)\tbput{$\sigg2\cc$}
\pcline{->}(41,0)(49,0)\tbput{$\ew$}
\pcline{->}(51,0)(59,0)\tbput{$\sigg1\cc$}
\pcline{->}(0,23)(0,1)\tlput{$\sigg1\aa$}
\pcline{->}(20,23)(20,13)\tlput{$\sigg1\xx$}
\pcline{->}(20,11)(20,1)\tlput{$\sigg2\aa$}
\pcline{->}(40,23)(40,13)\tlput{$\sigg1\xx$}
\pcline{->}(40,11)(40,7)\tlput{$\sigg2\yy$}
\pcline{->}(40,5)(40,1)\tlput{$\sigg3\aa$}
\pcline{->}(50,11)(50,7)\tlput{$\ew$}
\pcline{->}(50,5)(50,1)\tlput{$\sigg3\aa$}
\pcline{->}(60,23)(60,19)\trput{$\sigg1\yy$}
\pcline{->}(60,17)(60,13)\trput{$\sigg2\xx$}
\pcline{->}(60,11)(60,7)\trput{$\ew$}
\pcline{->}(60,5)(60,1)\trput{$\sigg3\aa$}
\end{picture}
\end{equation}
with $\xx \in \{\aa, \bb\}$ and $\yy \in \{\aa, \cc\} \cap \{\xx, \cc\}$. The equivalence of $\sigg1\aa \sigg2\bb \sigg3\aa$ and $\sigg1\yy \sigg2\xx \sigg3\aa$ would require in particular $\xx = \bb$. But, on the other hand, since no relation applies to the word $\sigg3\cc \sigg2\aa \sigg3\bb$, the equivalence class of the word $\sigg3\cc \sigg2\aa \sigg1\cc \sigg3\bb \sigg2\bb$ on the bottom edge of~\eqref{E:Case12BGL} with respect to the congruence generated by the relations of~\eqref{E:BGL} consists of two words only, namely $\sigg3\cc \sigg2\aa \sigg1\cc \sigg3\bb \sigg2\bb$ and $\sigg3\cc \sigg2\aa \sigg3\bb \sigg1\cc \sigg2\bb$, none of which begins with~$\sigg2\bb$. Hence no $(\sigg1\aa, \sigg2\bb\sigg3\cc \sigg2\cc)$-grid may be equivalent to~\eqref{E:Case12BGL}.
\end{proof}

The above negative result does not say that the monoid~$\BPr{\nn, \CC}$ is not cancellative, it just says that the criterion of Proposition~\ref{P:Crit} fails to apply. The proof of Fact~\ref{F:BGL} provides an explicit example of a valid relation that cannot be checked using reversing, namely
\begin{equation}\label{E:Counter}
\sigg2\bb \sigg3\cc \sigg2\cc \sigg1\aa \sigg2\bb \sigg3\aa \eqR \sigg1\aa \sigg3\cc \sigg2\aa \sigg1\cc \sigg3\bb \sigg2\bb,
\end{equation}
whose only proof requires introducing an intermediate word beginning with~$\sigg3\cc$, for instance $\sigg3\cc \sigg2\cc \sigg3\bb \sigg1\aa \sigg2\bb \sigg3\aa$. In other words, every van Kampen diagram witnessing~\eqref{E:Counter} must contain a vertex from which three edges start. By adding~\eqref{E:Counter} as a new (redundant) relation in the presentation, we can make the above relation eligible for factor reversing, but new obstructions are likely to appear, and it is not clear why the completion procedure thus sketched should come to an end. Thus, the following question is left open:

\begin{ques}\cite{BGL}\label{Q:BGL}
Does the monoid~$\BPr{\nn, \CC}$ admit cancellation? Does it embed in its universal group?
\end{ques}

\end{document}